\documentclass{amsart}

\usepackage[draft=false,pagebackref]{hyperref}
\usepackage{esint,amssymb,bm,mathtools}

\usepackage{xcolor}
\newtheorem{thm}[equation]{Theorem}
\newtheorem{lem}[equation]{Lemma}
\newtheorem{cor}[equation]{Corollary}
\theoremstyle{definition}
\newtheorem{defn}[equation]{Definition}
\newtheorem{rmk}[equation]{Remark}

\numberwithin{equation}{section}

\newcommand\abs[2][empty]{\csname#1\endcsname \lvert{#2}\csname#1\endcsname\rvert}
\newcommand\doublebar[2][empty]{\csname#1\endcsname \lVert{#2}\csname#1\endcsname\rVert}

\newcommand\mat[1]{\bm{#1}}
\newcommand\arr[1]{\bm{\dot{#1}}}

\newcommand\dist{\mathop{\mathrm{dist}}\nolimits}
\newcommand\Div{\mathop{\mathrm{div}}\nolimits}
\newcommand\Tr{\mathop{\smash{\arr{\mathrm{Tr}}}\vphantom{T}}\nolimits}
\newcommand\Trace{\mathop{\mathrm{Tr}}\nolimits}
\newcommand\M{\mathop{\smash{\arr{\mathrm{M}}}\vphantom{M}}\nolimits}

\newcommand\re{\mathop{\mathrm{Re}}\nolimits}

 \let\R\RR
 \let\C\CC
 
\newcommand\NN{\mathbb{N}} \let\N\NN
\newcommand\1{\mathbf{1}}
\newcommand\D{\mathcal{D}}
\newcommand\s{\mathcal{S}}

\def\XX{\mathfrak{X}}
\def\DD{\mathfrak{D}}
\def\NN{\mathfrak{N}}

\newcommand\pureH{\parallel}

\newcommand\dmn{{n+1}}
\newcommand\pdmn{{(n+1)}}
\newcommand\dmnMinusOne{n}

\begingroup
\makeatletter
\def\citation#1{}
\def\bibcite#1#2{}
\input Rellich-potentials.labels
\input Rellich-traces.labels
\global\def\eqref{\@ifstar\@eqref\@@eqref}
\global\def\@eqref#1{\textup{\tagform@{\ref*{#1}}}}
\global\def\@@eqref#1{\textup{\tagform@{\ref{#1}}}}
\endgroup
\begin{document}

\title[The higher order Neumann problem]{The Neumann problem for higher order elliptic equations with symmetric coefficients}

\author{Ariel Barton}
\address{Ariel Barton, Department of Mathematical Sciences,
			309 SCEN,
			University of Ar\-kan\-sas,
			Fayetteville, AR 72701}
\email{aeb019@uark.edu}

\author{Steve Hofmann}
\address{Steve Hofmann, 202 Math Sciences Bldg., University of Missouri, Columbia, MO 65211}
\email{hofmanns@missouri.edu}
\thanks{Steve Hofmann is partially supported by the NSF grant DMS-1361701.}

\author{Svitlana Mayboroda}
\address{Svitlana Mayboroda, Department of Mathematics, University of Minnesota, Minneapolis, Minnesota 55455}
\email{svitlana@math.umn.edu}
\thanks{Svitlana Mayboroda is partially supported by the NSF CAREER Award DMS 1056004, the NSF INSPIRE Award DMS 1344235, and the NSF Materials Research Science and Engineering Center Seed Grant.}

\begin{abstract}
In this paper we establish well posedness of the Neumann problem with boundary data in $L^2$ or the Sobolev space $\dot W^2_{-1}$, in the half space, for linear elliptic differential operators with coefficients that are constant in the vertical direction and in addition are self adjoint. This generalizes the well known well-posedness result of the second order case and is based on a higher order and one sided version of the classic Rellich identity, and is the first known well posedness result for a higher order operator with rough variable coefficients and boundary data in a Lebesgue or Sobolev space.
\end{abstract}



\maketitle 

\tableofcontents

\section{Introduction}

In this paper we will establish well posedness of the Neumann problem in the half-space $\R^\dmn_+$, with boundary data in the Lebesgue space $L^2(\R^n)$ or the Sobolev space $\dot W^2_{-1}(\R^n)$, for certain elliptic differential operators of the form
\begin{equation}\label{C:eqn:divergence}
Lu = (-1)^m \sum_{\abs{\alpha}=\abs{\beta}=m} \partial^\alpha (A_{\alpha\beta} \partial^\beta u).\end{equation}
Specifically, we will consider self-adjoint operators associated with coefficients $\mat A$ that are $t$-independent in the sense that
\begin{equation}\label{C:eqn:t-independent}\mat A(x,t)=\mat A(x,s)=\mat A(x) \quad\text{for all $x\in\R^n$ and all $s$, $t\in\R$}.\end{equation}

The Neumann problem has traditionally been regarded as more difficult than the Dirichlet problem. Indeed in two important cases, well posedness of the Dirichlet problem with boundary values in a Lebesgue or Sobolev space is known, but well posedness of the Neumann problem is not: in the case of second order operators with real $t$-independent coefficients \cite{HofKMP15A,HofKMP15B} and in the case of constant coefficient higher order operators in Lipschitz domains \cite{PipV95B,Ver96}.

In the case of higher order operators of the form~\eqref{C:eqn:divergence} with variable $t$-independent coefficients, we can bound Dirichlet boundary values  in a way that we cannot at present bound Neumann boundary values. See Theorem~\ref{C:thm:rellich} below. We will use good behavior of Dirichlet boundary values to establish well posedness of the Neumann problem; we cannot at present use the same arguments to establish well posedness results for the Dirichlet problem because we lack corresponding bounds on the Neumann boundary values. See Remark~\ref{C:rmk:Dirichlet:good}.

Indeed, even formulating the higher order Neumann problem is a difficult matter. Recall that in the second order case, the Neumann boundary value of a solution $u$ to $-\Div \mat A\nabla u=0$ is the conormal derivative $\nu\cdot \mat A\nabla u=0$, where $\nu$ is the unit outward normal derivative; this is preferred to the normal derivative $\nu\cdot \nabla u$ because, by the divergence theorem, we have a weak formulation
\begin{equation}
\label{C:eqn:Neumann:2}
\int_{\partial\Omega} \varphi\,\nu\cdot \mat A\nabla u\,d\sigma = \int_\Omega \nabla\varphi\cdot \mat A\nabla u \quad\text{for all $\varphi\in C^\infty_0(\R^\dmn)$.}\end{equation}
Some complexities are already apparent: it is often the case that an operator $L$ may be written $L=-\Div \mat A\nabla$ for more than one choice of coefficient matrix~$\mat A$, and different choices of coefficients $\mat A$ lead to different boundary values $\nu\cdot \mat A\nabla u$.

In the second order case, we can often eliminate this ambiguity, for example by requiring that $\mat A$ be self-adjoint. In the second order self-adjoint case the $L^2$-Neumann problem with $t$-independent coefficients is well posed; see \cite{KenP93}. If $\mat A$ is not self-adjoint, we still do not know whether the Neumann problem is well posed, even for second order operators with real $t$-independent coefficients.

In the higher order case, $L$ may be associated to multiple self-adjoint coefficient matrices~$\mat A$. For example, the biharmonic operator $L=\Delta^2$ may be associated with coefficients $\mat A$ such that
\[\sum_{\abs\alpha=\abs\beta=2}\partial^\alpha v\,A_{\alpha\beta}\,\partial^\beta w=  \rho \Delta v \Delta w + (1-\rho) \sum_{j=1}^\dmn \sum_{k=1}^\dmn \partial^2_{jk} v\,\partial^2_{jk} w\]
for any real number~$\rho$. Notably, the Neumann problem for the biharmonic operator (as studied in \cite{Ver05}) is well posed for some values of the parameter $\rho$ and ill posed for others. Thus, we will establish well posedness of the Neumann problem under a boundary ellipticity condition (the bound~\eqref{C:eqn:elliptic:slices:strong} below) that is somewhat more restrictive than the ellipticity condition~\eqref{C:eqn:elliptic} standard in the theory.

Another complication arises in generalizing the formulation~\eqref{C:eqn:Neumann:2} to the higher order case.
Notice the appearance of the Dirichlet boundary values $\varphi\big\vert_{\partial\Omega}$ of $\varphi$ on the left-hand side of formula~\eqref{C:eqn:Neumann:2}. The Neumann boundary values are then dual to the Dirichlet boundary values. Thus, different formulations of the Dirichlet problem lead to different formulations of the Neumann problem. If we let the Dirichlet boundary values of $\varphi$ be $(\varphi,\partial_\nu\varphi,\dots,\partial_\nu^{m-1}\varphi)$, where $\partial_\nu$ denotes the partial derivative in the normal direction, then a straightforward (if tedious) integration by parts yields an analogue to formula~\eqref{C:eqn:Neumann:2} from which the Neumann boundary values may be extracted. See \cite[formula~(1.1.1)]{CohG85}, \cite{Ver05} or \cite[Proposition~4.3]{MitM13A}. 

However, it is often convenient to regard $\nabla^{m-1} \varphi\big\vert_{\partial\Omega}$ as the Dirichlet boundary values of~$\varphi$: the various components of $\nabla^{m-1}\varphi\big\vert_{\partial\Omega}$ may reasonably be expected to all possess the same degree of smoothness, while the lower order derivatives $\varphi$, $\partial_\nu\varphi,\dots,\partial_\nu^{m-2}\varphi$ appearing above may be expected to possess further orders of smoothness. See \cite{BarHM15p,BarHM17pB}. This is the formulation we shall use in the present paper. However, this does yield some additional complications, which we shall discuss momentarily.

We now discuss the details of our formulation of Neumann boundary values. 
If $\varphi$ is smooth and compactly supported in $\R^\dmn$, and if $Lu=0$ in $\Omega\subset\R^\dmn$, where $\partial\Omega$ is connected, and where $\nabla^m u$ is locally integrable up to the boundary, then 
\begin{equation}\label{C:eqn:introduction:Neumann:LHS}
\sum_{\abs\alpha=\abs\beta=m} \int_\Omega \partial^\alpha \varphi\,A_{\alpha\beta}\,\partial^\beta u\end{equation}
depends only on the behavior of $\varphi$ near~$\partial\Omega$, and in particular depends only on $\nabla^{m-1}\varphi\big\vert_{\partial\Omega}$. We denote the Neumann boundary values of $u$ by $\M_{\mat A} u$ and say that 
\begin{equation}\label{C:eqn:Neumann:intro}
\M_{\mat A}u=\arr g\quad\text{if }\sum_{\abs\alpha=\abs\beta=m} \int_\Omega \partial^\alpha \varphi\,A_{\alpha\beta}\,\partial^\beta u = \sum_{\abs\gamma=m-1} \int_{\partial\Omega} \partial^\gamma \varphi\,g_\gamma\,d\sigma
\end{equation}
for all $\varphi\in C^\infty_0(\R^\dmn)$.

We remark that $\M_{\mat A} u$ is an operator on $\{\nabla^{m-1}\varphi\big\vert_{\partial\Omega}:\varphi\in C^\infty_0(\R^\dmn)\}$. This is a \emph{proper} subspace of the set of all arrays of smooth, compactly supported functions defined in a neighborhood of~$\partial\Omega$. Thus, $\M_{\mat A} u$ is not an array of distributions; it is an equivalence class of distributions modulo arrays $\arr n$ that satisfy $\sum_{\abs\gamma=m-1} \int_{\partial\Omega} \partial^\gamma \varphi\,n_\gamma\,d\sigma=0$ for all $\varphi\in C^\infty_0(\R^\dmn)$. If $\arr g$ is an array of distributions (or functions) defined on~$\partial\Omega$, then the expression $\M_{\mat A} u=\arr g$ represents a slight abuse of notation; we mean that $\arr g$ is a representative of the equivalence class of distributions $\M_{\mat A} u$. This complication could be avoided by writing the right-hand side as $\sum_{j=0}^{m-1}\int_{\partial\Omega} \partial_\nu^j \varphi\, g_j\,d\sigma$, but then (as mentioned above) the various components $g_j$ of $\arr g$ would need to possess different orders of smoothness.

Our first well posedness result is the following theorem.

\begin{thm}\label{C:thm:well-posed:2}
Suppose that $L$ is an elliptic operator of the form \eqref{C:eqn:divergence} of order~$2m$, associated with coefficients $\mat A$ that are $t$-independent in the sense of formula~\eqref{C:eqn:t-independent} and are bounded in the sense of satisfying the bound~\eqref{C:eqn:elliptic:bounded}.

Suppose in addition that $A$ satisfies the boundary ellipticity condition \eqref{C:eqn:elliptic:slices:strong} and is self-adjoint, that is, that $A_{\alpha\beta}(x)=\overline{A_{\beta\alpha}(x)}$ for any $x\in\R^n$ and any $\abs\alpha=\abs\beta=m$.

For each $\arr g\in L^2(\R^\dmnMinusOne)$ there is a solution to the Neumann problem with boundary data $\arr g$, that is, a function $w$ defined in $\R^\dmn_+$ that satisfies
\begin{equation}
\label{C:eqn:neumann:regular}
\left\{\begin{gathered}
\begin{aligned}
Lw&=0 \text{ in }\R^\dmn_+
,\\
\M_{\mat A} w &= \arr g
,\end{aligned}
\\
\int_{\R^\dmnMinusOne}\int_0^\infty \abs{\nabla^m\partial_t w(x,t)}^2 \,t\,dt\,dx
+
\sup_{t>0}\doublebar{\nabla^m w(\,\cdot\,,t)}_{ L^2(\R^n)}^2
\leq C\doublebar{\arr g}_{L^2(\R^\dmnMinusOne)}^2
.\end{gathered}\right.\end{equation}
The solution $w$ is unique up to adding polynomials of degree~$m-1$.
\end{thm}

In a forthcoming paper \cite{BarHM17pD}, we intend to show that the solutions~$w$, in addition to satisfying square-function and uniform $L^2$ estimates, also satisfy nontangential maximal estimates.

It is common in the theory of divergence form equations to consider two forms of the Dirichlet problem. The first is the Dirichlet problem with boundary data in $L^2$ or (more generally) in a Lebesgue space~$L^p$. The second is the Dirichlet regularity problem, that is, the Dirichlet problem with boundary data in a boundary Sobolev space. For example, if the matrix $\mat A$ in formula~\eqref{B:eqn:divergence} has constant coefficients, and if $\Omega\subset\R^\dmn$ is a bounded Lipschitz domain, then by \cite{PipV95B} and \cite{DahKPV97} the Dirichlet problem
\begin{gather*}
Lv=0\text{ in }\Omega,\quad \nabla^{m-1}v\big\vert_{\partial\Omega}=\arr f, \quad \int_{\Omega} \abs{\nabla^m v(X)}^2\,\dist(X,\partial\Omega)\,dX
\leq C \doublebar{\arr f}_{L^2(\partial\Omega)}
\end{gather*}
and the regularity problem
\begin{gather*}
Lw=0\text{ in }\Omega,\quad \nabla^{m-1}w\big\vert_{\partial\Omega}=\arr f, \quad 
\int_{\Omega} \abs{\nabla^{m+1} w(X)}^2\,\dist(X,\partial\Omega)\,dX
\leq C \doublebar{\arr f}_{\dot W_1^2(\partial\Omega)}
\end{gather*}
are well posed. Here $\dot W^p_1(\partial\Omega)$ is the boundary Sobolev space of functions whose tangential gradient lies in~$L^p(\partial\Omega)$. Notice that the estimates on the solution $w$ in the regularity problem just given are very similar to those of Theorem~\ref{C:thm:well-posed:2}. 

The Neumann problem is most often studied in the case where the boundary data lies in a Lebesgue space; generally, the Neumann problem then has the same sorts of estimates as the regularity problem. However, it is also possible to study the Neumann problem with boundary data in a negative smoothness space, that is, the dual space to a Sobolev space; the solutions then have the same sort of estimates as the Dirichlet problem.

However, generalizing the weak formulation~\eqref{C:eqn:Neumann:intro} of Neumann boundary values is somewhat problematic given such estimates on~$v$. Recall that we seek solutions $v$ that satisfy 
\begin{equation}\label{C:eqn:rough:estimate}
\int_{\R^\dmnMinusOne}\int_0^\infty \abs{\nabla^m v(x,t)}^2 \,t\,dt\,dx
<\infty.
\end{equation}
For such $v$ the gradient $\nabla^m v$ is not locally integrable up to the boundary of $\R^\dmn_+$, and so the integral~\eqref{C:eqn:introduction:Neumann:LHS} does not converge absolutely and the formula~\eqref{C:eqn:Neumann:intro} for $\M_{\mat A} v$ may not be meaningful. There are several ways to resolve this difficulty. First, one may consider solutions $v$ that satisfy $\nabla^m v\in L^2(\R^\dmn_+)$ as well as the estimate~\eqref{C:eqn:rough:estimate}; for such~$v$ the integral~\eqref{C:eqn:introduction:Neumann:LHS} converges for all $\varphi$ smooth and compactly supported (and indeed for all $\varphi$ with $\nabla^m \varphi\in L^2(\R^\dmn)$).

Second, given an array of test functions $\nabla^{m-1}\varphi\big\vert_{\partial\R^\dmn_+}$, one may define the pairing $\langle \nabla^{m-1}\varphi\big\vert_{\partial\R^\dmn_+},\M_{\mat A} v\rangle$ in terms of a specific extension, that is, a particular choice of function $\mathcal{E}\varphi$ with $\nabla^{m-1}\mathcal{E}\varphi\big\vert_{\partial\R^\dmn_+}=\nabla^{m-1}\varphi\big\vert_{\partial\R^\dmn_+}$. We will use the following extension.
Suppose that $\varphi$ is smooth and compactly supported in $\R^\dmn$. Let $\varphi_k(x)=\partial_\dmn^k\varphi(x,0)$.
If $t\in\R$, let
\begin{equation}
\label{C:eqn:Neumann:extension}
\mathcal{E}\varphi(x,t) = 
\sum_{k=0}^{m-1}\frac{1}{k!}t^k \mathcal{Q}^m_t \varphi_k(x)
\quad\text{where}\quad\mathcal{Q}_t^m = e^{-(-t^2\Delta_\pureH)^m}.\end{equation}
Here $\Delta_\pureH$ is the Laplacian taken purely in the horizontal variables.
Observe that $\mathcal{E}\varphi$ is also smooth on $\overline{\R^\dmn_\pm}$ up to the boundary, albeit is not compactly supported, and that $\nabla^{m-1}\mathcal{E}\varphi(x,0)=\nabla^{m-1}\varphi(x,0)$.

In \cite[Theorem~\ref*{B:thm:Neumann:1}]{BarHM17pB} it was shown that, if $v$ satisfies the bound \eqref{C:eqn:rough:estimate}, then the integral
$ \int_{\R^n} \partial^\alpha \mathcal{E}\varphi(x,t)\,A_{\alpha\beta}(x)\,\partial^\beta v(x,t)\,dx$ converges absolutely for any \emph{fixed} $t>0$ (and that the value of this integral is continuous in~$t$), and that 
\[\lim_{\varepsilon\to 0^+} \lim_{T\to\infty} \smash{\sum_{\substack{\abs\alpha=\abs\beta=m}}}\int_\varepsilon^T\int_{\R^n} \partial^\alpha \mathcal{E}\varphi(x,t)\,A_{\alpha\beta}(x)\,\partial^\beta v(x,t)\,dx\,dt\]
exists and equals a number whose absolute value is at most
\[
C\doublebar{\nabla^{m-1}\varphi(\,\cdot\,,0)}_{L^2(\R^n)} \biggl(\int_{\R^\dmnMinusOne}\int_0^\infty \abs{\nabla^m v(x,t)}^2 \,t\,dt\,dx\biggr)^{1/2}.\]
Thus, for such $v$ one may define the Neumann boundary values using the  extension~$\mathcal{E}$. See formula~\eqref{C:eqn:Neumann:E} below.

By \cite[Lemma~\ref*{B:lem:Neumann:W2}]{BarHM17pB}, if $\nabla^m v\in L^2(\R^\dmn_+)$, then the two definitions of Neumann boundary values coincide, and so the two ways of studying Neumann boundary values of rough solutions are equivalent.

The second main theorem of this work, to be proven via duality with Theorem~\ref{C:thm:well-posed:2}, is as follows.

\begin{thm}\label{C:thm:well-posed:1} Let $L$ be as in Theorem~\ref{C:thm:well-posed:2}. Then for each array $\arr g$ of bounded linear operators on $\dot W^2_1(\R^n)$, there is a solution to the rough Neumann problem with boundary data $\arr g$, that is, a function $v$ defined in $\R^\dmn_+$ that satisfies
\begin{equation}
\label{C:eqn:neumann:rough}
\left\{\begin{gathered}
\begin{aligned}
Lv&=0 \text{ in }\R^\dmn_+
,\\
\M_{\mat A} v &= \arr g
,\end{aligned}
\\
\int_{\R^\dmnMinusOne}\int_0^\infty \abs{\nabla^m v(x,t)}^2 \,t\,dt\,dx
\leq C\doublebar{\arr g}_{\dot W^2_{-1}(\R^n)}
\end{gathered}\right.\end{equation}
where $\M_{\mat A} v$  is defined in terms of a distinguished extension as above.
The solution $v$ is unique up to adding polynomials of degree~$m-1$.
\end{thm}

We also have a perturbative result.
\begin{thm}\label{C:thm:perturbation} Suppose that $L_0$ is an elliptic operator of the form \eqref{C:eqn:divergence} of order~$2m$, associated with coefficients $\mat A_0$ that are $t$-independent in the sense of formula~\eqref{C:eqn:t-independent} and are bounded and elliptic in the sense of satisfying the bounds~\eqref{C:eqn:elliptic:bounded} and~\eqref{C:eqn:elliptic}.

Suppose that the Neumann problem \eqref{C:eqn:neumann:regular} for $\mat A_0$ is well posed; that is, for every $\arr g\in L^2(\R^n)$ there is a solution $w_0$ to the problem \eqref{C:eqn:neumann:regular} with $\mat A$ replaced by $\mat A_0$, and that $w_0$ is unique up to adding polynomials of degree $m-1$. Suppose that the corresponding problem in the lower half-space
\begin{equation}
\label{C:eqn:neumann:regular:lower}
\left\{\begin{gathered}
\begin{aligned}
L_0w&=0 \text{ in }\R^\dmn_-
,\\
\M_{\mat A_0}^- w &= \arr g
,\end{aligned}
\\
\int_{\R^\dmnMinusOne}\int_{-\infty}^0 \abs{\nabla^m\partial_t w(x,t)}^2 \,\abs{t}\,dt\,dx
+
\sup_{t<0}\doublebar{\nabla^m w(\,\cdot\,,t)}_{ L^2(\R^n)}^2
\leq C\doublebar{\arr g}_{L^2(\R^\dmnMinusOne)}^2
\end{gathered}\right.\end{equation}
is also well posed.

Then there is some $\varepsilon>0$, depending only on the ellipticity constants $\lambda$ and $\Lambda$ in formulas~\eqref{C:eqn:elliptic} and \eqref{C:eqn:elliptic:bounded} and the constants $C$ in the problems~\eqref{C:eqn:neumann:regular} and~\eqref{C:eqn:neumann:regular:lower}, such that if $\mat A$ is $t$-independent and $\doublebar{\mat A-\mat A_0}_{L^\infty(\R^n)}<\varepsilon$, then the Neumann problem \eqref{C:eqn:neumann:regular} is well posed for coefficients~$\mat A$.

Similarly, if the rough Neumann problem \eqref{C:eqn:neumann:rough} for coefficients $\mat A_0$ is well posed in both $\R^\dmn_+$ and $\R^\dmn_-$, and if $\doublebar{\mat A-\mat A_0}_{L^\infty(\R^n)}$ is small enough, then the rough Neumann problem \eqref{C:eqn:neumann:rough} is well posed for coefficients~$\mat A$ as well.


\end{thm}

Notice that Theorems~\ref{C:thm:well-posed:2} and~\ref{C:thm:well-posed:1} concern only operators with self-adjoint coefficients, while Theorem~\ref{C:thm:perturbation} concerns arbitrary (non-self-adjoint) $t$-independent coefficients. In particular, combining these three results gives the following corollary.

\begin{cor} \label{C:cor:almost}
Fix some $\Lambda>\lambda>0$ and some positive integer~$m$. Then there is some $\varepsilon>0$, depending only on the dimension $\dmn$ and the constants $\Lambda$, $\lambda$ and~$m$, with the following significance.

Suppose that $L$ is an elliptic operator of the form \eqref{C:eqn:divergence} of order~$2m$, associated with coefficients $\mat A$ that are $t$-independent in the sense of formula~\eqref{C:eqn:t-independent} and are bounded and elliptic in the sense of satisfying the bounds~\eqref{C:eqn:elliptic:bounded} and~\eqref{C:eqn:elliptic}.

Let $A^*_{\alpha\beta}=\overline{A_{\beta\alpha}}$.
Suppose further that $\doublebar{\mat A-\mat A^*}_{L^\infty(\R^n)}<\varepsilon$.

Then the Neumann problems \eqref{C:eqn:neumann:regular} and the rough Neumann problem~\eqref{C:eqn:neumann:rough} are well posed for the coefficients~$\mat A$.
\end{cor}
The $\varepsilon=0$ case of this corollary is  Theorem~\ref{C:thm:well-posed:2} or~\ref{C:thm:well-posed:1}; by letting $\mat A_1=\mat A$ and letting $\mat A_0$ be a nearby self adjoint matrix (for example, $\mat A_0=\frac{1}{2}\mat A+\frac{1}{2}\mat A^*$), we obtain Corollary~\ref{C:cor:almost} from Theorem~\ref{C:thm:perturbation}.

We now turn to the history of the Neumann problem. We begin with the case of second-order operators, and in particular with harmonic functions (that is, the case $L=-\Delta$). In \cite{JerK81B} Jerison and Kenig established well posedness of the Neumann problem for harmonic functions in Lipschitz domains with $L^2$ boundary data. (They established well posedness with nontangential maximal estimates, not the square-function estimates used in this paper; however, as shown in \cite{Dah80A}, for harmonic functions the two estimates are equivalent.) This was extended to $L^p$ boundary data for $1<p<2+\varepsilon$ in \cite{DahK87}.  Here $\varepsilon$ is a (possibly small) positive number that depends on the Lipschitz character of the domain under consideration.

Turning to more general second order operators, in \cite{KenP93} Kenig and Pipher established well posedness of the $L^p$-Neumann problem (with nontangential estimates), $1<p<2+\varepsilon$, for solutions to $\Div \mat A\nabla u=0$, where $\mat A$ is a real symmetric radially constant matrix, in the unit ball. The same arguments yield well posedness of the Neumann problem for real symmetric $t$-independent coefficients in the upper half-space. (In the case of second-order operators, but not higher order operators, a straightforward change of variables argument allows an immediate generalization from results for radially independent coefficients in the unit ball to radially independent coefficients in starlike Lipschitz domains, or from results for $t$-independent coefficients in the half-space to $t$-independent coefficients in Lipschitz graph domains.)
Again, for $t$-independent coefficients in the second order case, the square function estimates used in this paper can often be shown to be equivalent to the nontangential estimates common in the theory; see \cite{DahJK84}, \cite[Theorem~1.7]{HofKMP15A} and \cite[Theorem~2.3]{AusA11}.

The Neumann problem is known to be well posed for a few other special classes of second order operators.
In two dimensions the $L^p$-Neumann problem is well posed for real nonsymmetric $t$-independent coefficients in the upper half-plane provided $1<p<1+\varepsilon$; see \cite{KenR09}. If $\mat A$ is of block form (that is, if $A_{j\pdmn}=A_{\pdmn j}=0$ for all $1\leq j\leq n$), then well posedness of the Neumann problem in the half-space follows from the positive resolution of the Kato square root conjecture \cite{AusHLMT02}; see \cite[Remark~2.5.6]{Ken94}. (The result \cite{KenR09} for real coefficients is preserved under a change of variables and so is also valid in Lipschitz graph domains, but the block form is not preserved by a change of variables and so is not known to generalize to Lipschitz domains.)

We may also consider perturbation results for $t$-independent coefficients.
If $\mat A$ is $t$-independent and if $\doublebar{\mat A-\mat A_0}_{L^\infty}$ is small enough, for some $t$-independent matrix $\mat A_0$ that is real symmetric (or complex and self-adjoint), of block form, or constant, then the $L^2$-Neumann problem for $\Div \mat A\nabla$ is well posed in the half-space; see \cite{AusAH08}, or \cite{AlfAAHK11} under a few additional assumptions. If $\mat A_0$ is an arbitrary $t$-independent coefficient matrix for which the $L^2$-Neumann problem is well posed, then the $L^2$-Neumann problem for $\mat A$ is well posed; see \cite{AusAM10}, or again \cite{AlfAAHK11} under some additional assumptions. If $\mat A_0$ is real symmetric, then by \cite{HofMitMor15} the $L^p$-Neumann problem is well posed for $\mat A$ provided $1<p<2+\varepsilon$. (In fact, they showed that well posedness extends to the range $1-\varepsilon<p\leq 1$ if we consider boundary data in the Hardy space $H^p$ rather than the Lebesgue space~$L^p$.) In two dimensions, if $\mat A_0$ is real but not symmetric (that is, if $\mat A_0$ is as in \cite{KenR09}), then by \cite{Bar13} the $L^p$-Neumann problem is well posed for $1<p<1+\varepsilon$.

The $t$-independent case may be viewed as a starting point for certain $t$-dependent perturbations; see \cite{KenP93,KenP95,AusA11,AusR12,HofMayMou15}.

Very few results are known concerning well posedness of the higher order Neumann problem with boundary data in a Lebesgue space. Some results are available in the case of the biharmonic operator $\Delta^2$. In particular, the $L^p$-Neumann problem for $1<p<\infty$ was shown to be well posed in $C^1$ domains in $\R^2$ in \cite{CohG85}, and in domains of arbitrary dimension whose unit outward normal lies in $VMO$ in~\cite{MitM13B}. Turning to the case of Lipschitz domains $\Omega\subset\R^\dmn$, the $L^p$-Neumann problem was shown to be well posed in \cite{Ver05} for $2-\varepsilon<p<2+\varepsilon$, and in \cite{She07B} for $\max(1,2n/(n+2)-\varepsilon)<p<2+\varepsilon$.

We now turn to the Neumann problem with  boundary data in negative smoothness spaces. Well posedness of the Neumann problem with boundary data in the \emph{fractional} negative smoothness space (the Besov space) $\dot B^2_{-1/2}(\partial\Omega)$ follows from the Lax-Milgram theorem. Well posedness of the Neumann problem with boundary data in the Besov space $\dot B^p_s(\partial\Omega)$ for certain values of $p$, $s$ with $-1<s<0$ and $0<p<\infty$ was established in \cite{FabMM98,Zan00,May05,MayMit04A} (for harmonic functions), in \cite{BarM16A} (for second-order operators with $t$-independent coefficients for which the $L^p$-Neumann problem is well posed, for example, for self-adjoint coefficients or for real coefficients in two dimensions), in \cite{MitM13B} (the biharmonic equation), \cite{MitM13A} (constant coefficient equations of order $2m$, for $m\geq 1$) and \cite{Bar16} (for arbitrary elliptic bounded measurable coefficients).

We conclude our discussion of the history of the Neumann problem with the case of boundary data in  the negative integer smoothness space $\dot W^p_{-1}(\R^n)$. For second-order $t$-independent operators $\Div \mat A\nabla$, this problem was investigated in \cite{AusM14,AusS14p}. In \cite{AusM14}, the $\dot W^p_{-1}(\R^n)$-Neumann problem was shown to be equivalent to the $L^{p'}$-Neumann problem for $\Div \mat A^*\nabla$, where $\mat A^*$ is the adjoint matrix; thus, in particular the $\dot W^2_{-1}$-Neumann problem is well posed for self adjoint coefficients, coefficients in block form, constant coefficients, or small $t$-independent $L^\infty$ perturbations thereof. \cite{AusS14p} treated the converse problem, that is, the problem of trace results for solutions $v$ that satisfy the bound \eqref{C:eqn:rough:estimate} or similar results, and thereby proved some further perturbative results.

We remark that the approach of \cite{AusM14,AusS14p} is similar to the approach of this paper. That is, let $\D^{\mat A}$ and $\s^L$ be the double and single layer potentials associated to our coefficients~$\mat A$ (to be defined in Section~\ref{C:sec:dfn:potentials}); we remark that these operators take as input arrays of functions or distributions $\arr f$ or $\arr g$ defined on $\R^n$ and return functions $\D^{\mat A}\arr f$ or $\s^L\arr g$ that satisfy $L(\D^{\mat A}\arr f)=0$ and $L(\s^L\arr g)=0$ in $\R^\dmn_+$ and $\R^\dmn_-$. If $u$ is a solution to $Lu=0$ in $\R^\dmn_+$, then let $\Tr_{m-1}^+u$ and $\M_{\mat A}^+$ denote the Dirichlet and Neumann boundary values of~$u$. Given certain estimates on~$u$ (see Section~\ref{C:sec:green}), we have the Green's formula
\begin{equation}\label{C:eqn:green:introduction}
u=-\D^{\mat A}(\Tr_{m-1}^+  u) +\s^{L}(\M_{\mat A}^+  u) \quad\text{in $\R^\dmn_+$}
.\end{equation}
Given bounds on $\D^{\mat A}$ and $\s^L$ established in \cite{BarHM15p,BarHM17pA} (see Section~\ref{C:sec:potentials:bounds}), we have the estimates
\begin{align}
\label{C:eqn:AusMS:1}
\int_{\R^\dmnMinusOne}\int_0^\infty \abs{\nabla^m v(x,t)}^2 \,t\,dt\,dx
&\leq C\doublebar{\Tr_{m-1}^+ v}_{L^2(\R^n)}^2+C\doublebar{\M_{\mat A}^+ v}_{\dot W^2_{-1}(\R^n)}^2
,\\
\label{C:eqn:AusMS:2}
\int_{\R^\dmnMinusOne}\int_0^\infty \abs{\nabla^m\partial_t w(x,t)}^2 \,t\,dt\,dx
&+\sup_{t>0}\doublebar{\nabla^m w(\,\cdot\,,t)}_{ L^2(\R^n)}^2
\\\nonumber
&\leq C\doublebar{\Tr_{m-1}^+ w}_{\dot W^2_1(\R^n)}^2 +C\doublebar{\M_{\mat A}^+ w}_{L^2(\R^\dmnMinusOne)}^2
.\end{align}
The trace results of \cite{BarHM17pB} (see Section~\ref{C:sec:trace}) give the reverse inequalities. We will exploit this equivalence of norms to prove well posedness. The approach of \cite{AusM14,AusS14p} is also to prove an equivalence between tent space estimates on a solution $u$ and certain norms of the Dirichlet and Neumann boundary values $u\big\vert_{\partial\R^\dmn_+}$ and $\nu\cdot \mat A\nabla u$. Their approach is mediated by semigroups rather than layer potentials; however, we remark that by \cite{Ros13} their semigroups are in some sense equivalent to layer potentials.

The outline of this paper is as follows.

In Section~\ref{C:sec:dfn} we will define our terminology. In particular, we will define the layer potentials $\D^{\mat A}$ and~$\s^L$. In Section~\ref{C:sec:known} we will summarize some known results: regularity of solutions to $Lu=0$ from \cite{Bar16} and \cite{AlfAAHK11,BarHM15p}, boundedness of layer potentials from \cite{BarHM17pA}, and trace results from \cite{BarHM17pB}, that is, bounds on the Dirichlet and Neumann boundary values of a solution $u$ to $Lu=0$. In Section~\ref{C:sec:boundary} we will prove some additional results concerning boundary values of solutions and of layer potentials, in particular the Green's formula~\eqref{C:eqn:green:introduction}.

In Section~\ref{C:sec:rellich} we will prove a one-sided version of the Rellich identity. This will allow us to control the Dirichlet boundary values of a solution $w$ to $Lw=0$ that satisfies the estimates given in the problem~\eqref{C:eqn:neumann:regular}. This combined with the estimate~\eqref{C:eqn:AusMS:2} establishes uniqueness of solutions $w$ to the Neumann problem \eqref{C:eqn:neumann:regular} and yields the estimate
\begin{equation*}\int_{\R^\dmnMinusOne}\int_0^\infty \abs{\nabla^m\partial_t w(x,t)}^2 \,t\,dt\,dx
+\sup_{t>0}\doublebar{\nabla^m w(\,\cdot\,,t)}_{ L^2(\R^n)}^2
\leq C\doublebar{\M_{\mat A}^+ w}_{L^2(\R^\dmnMinusOne)}^2
\end{equation*}
in problem~\eqref{C:eqn:neumann:regular}.

In Section~\ref{C:sec:special} we will show existence of solutions to the Neumann problem~\eqref{C:eqn:neumann:regular} for a particular choice of coefficients~$\mat A_0$, thus completing the proof of Theorem~\ref{C:thm:well-posed:2} for that choice of coefficients.

In order to prove Theorems~\ref{C:thm:well-posed:1} and~\ref{C:thm:perturbation}, we will need some additional properties of layer potentials (by now well known in the second order case and generalized to the higher order case in \cite{Bar17p}). This approach also provides a straightforward way to generalize Theorem~\ref{C:thm:well-posed:2} from the specific coefficients~$\mat A_0$ to full generality. We will state these results in Section~\ref{C:sec:invertible:known} and apply them in Section~\ref{C:sec:invertible:application}.

\subsection*{Acknowledgements}
We would like to thank the American Institute of Mathematics for hosting the SQuaRE workshop on ``Singular integral operators and solvability of boundary problems for elliptic equations with rough coefficients,'' and the Mathematical Sciences Research Institute for hosting a Program on Harmonic Analysis,  at which many of the results and techniques of this paper were discussed.

\section{Definitions}
\label{C:sec:dfn}

In this section, we will provide precise definitions of the notation and concepts used throughout this paper. 

We mention that throughout this paper, we will work with elliptic operators~$L$ of order~$2m$ in the divergence form \eqref{C:eqn:divergence} acting on functions defined on~$\R^\dmn$.
We let $\R^\dmn_+$ and $\R^\dmn_-$ denote the upper and lower half-spaces $\R^n\times (0,\infty)$ and $\R^n\times(-\infty,0)$; we will identify $\R^n$ with $\partial\R^\dmn_\pm$.


\subsection{Multiindices and arrays of functions}

We will reserve the letters $\alpha$, $\beta$, $\gamma$, $\zeta$ and~$\xi$ to denote multiindices in $\N^\dmn$. (Here $\N$ denotes the nonnegative integers.) If $\zeta=(\zeta_1,\zeta_2,\dots,\zeta_\dmn)$ is a multiindex, then we define $\abs{\zeta}$, $\partial^\zeta$ 
in the usual ways, as $\abs{\zeta}=\zeta_1+\zeta_2+\dots+\zeta_\dmn$, $\partial^\zeta=\partial_{x_1}^{\zeta_1}\partial_{x_2}^{\zeta_2} \cdots\partial_{x_\dmn}^{\zeta_\dmn}$. 

We will routinely deal with arrays $\arr F=\begin{pmatrix}F_{\zeta}\end{pmatrix}$ of numbers or functions indexed by multiindices~$\zeta$ with $\abs{\zeta}=k$ for some~$k\geq 0$.
In particular, if $\varphi$ is a function with weak derivatives of order up to~$k$, then we view $\nabla^k\varphi$ as such an array.

The inner product of two such arrays of numbers $\arr F$ and $\arr G$ is given by
\begin{equation*}\bigl\langle \arr F,\arr G\bigr\rangle =
\sum_{\abs{\zeta}=k}
\overline{F_{\zeta}}\, G_{\zeta}.\end{equation*}
If $\arr F$ and $\arr G$ are two arrays of functions defined in a set $\Omega$ in Euclidean space, then the inner product of $\arr F$ and $\arr G$ is given by
\begin{equation*}\bigl\langle \arr F,\arr G\bigr\rangle_\Omega =
\sum_{\abs{\zeta}=k}
\int_{\Omega} \overline{F_{\zeta}(X)}\, G_{\zeta}(X)\,dX.\end{equation*}

We let $\vec e_j$ be the unit vector in $\R^\dmn$ in the $j$th direction; notice that $\vec e_j$ is a multiindex with $\abs{\vec e_j}=1$. We let $\arr e_{\zeta}$ be the ``unit array'' corresponding to the multiindex~$\zeta$; thus, $\langle \arr e_{\zeta},\arr F\rangle = F_{\zeta}$.

We will let $\nabla_\pureH$ denote either the gradient in~$\R^n$, or the $n$ horizontal components of the full gradient~$\nabla$ in $\R^\dmn$. (Because we identify $\R^n$ with $\partial\R^\dmn_\pm\subset\R^\dmn$, the two uses are equivalent.) If $\zeta$ is a multiindex with $\zeta_\dmn=0$, we will occasionally use the terminology $\partial_\pureH^\zeta$ to emphasize that the derivatives are taken purely in the horizontal directions.

\subsection{Elliptic differential operators and their bounds}

Let $\mat A = \begin{pmatrix} A_{\alpha\beta} \end{pmatrix}$ be a matrix of measurable coefficients defined on $\R^\dmn$, indexed by multtiindices $\alpha$, $\beta$ with $\abs{\alpha}=\abs{\beta}=m$. If $\arr F$ is an array, then $\mat A\arr F$ is the array given by
\begin{equation*}(\mat A\arr F)_{\alpha} =
\sum_{\abs{\beta}=m}
A_{\alpha\beta} F_{\beta}.\end{equation*}

We will consider coefficients that satisfy the G\r{a}rding inequality
\begin{align}
\label{C:eqn:elliptic}
\re {\bigl\langle\nabla^m \varphi,\mat A\nabla^m \varphi\bigr\rangle_{\R^\dmn}}
&\geq
	\lambda\doublebar{\nabla^m\varphi}_{L^2(\R^\dmn)}^2
	\quad\text{for all $\varphi\in\dot W^2_m(\R^\dmn)$}
\end{align}
and the bound
\begin{align}
\label{C:eqn:elliptic:bounded}
\doublebar{\mat A}_{L^\infty(\R^\dmn)}
&\leq
	\Lambda
\end{align}
for some $\Lambda>\lambda>0$.
In this paper we will focus exclusively on coefficients that are $t$-inde\-pen\-dent, that is, that satisfy formula~\eqref{C:eqn:t-independent}.

We let $L$ be the $2m$th-order divergence-form operator associated with~$\mat A$. That is, we say that $L u=0$ in~$\Omega$ in the weak sense if, for every $\varphi$ smooth and compactly supported in~$\Omega$, we have that
\begin{equation}
\label{C:eqn:L}
\bigl\langle\nabla^m\varphi, \mat A\nabla^m u\bigr\rangle_\Omega
=\sum_{\abs{\alpha}=\abs{\beta}=m}
\int_{\Omega}\partial^\alpha \bar \varphi\, A_{\alpha\beta}\,\partial^\beta u
=
0
.
\end{equation}

Throughout the paper we will let $C$ denote a constant whose value may change from line to line, but which depends only on the dimension $\dmn$, the ellipticity constants $\lambda$ and $\Lambda$ in the bounds \eqref{C:eqn:elliptic} and~\eqref{C:eqn:elliptic:bounded}, and the order~$2m$ of our elliptic operators. Any other dependencies will be indicated explicitly.

We will need a stronger ellipticity condition. 
Notice that if $\mat A$ is $t$-indepen\-dent, then the bound \eqref{C:eqn:elliptic} implies that if $\varphi$ is constant in the $t$-direction, then
\begin{equation}
\label{C:eqn:elliptic:slices:weak}
\re\langle \nabla^m \varphi,\mat A\nabla^m\varphi\rangle_{\R^n} 
\geq \lambda\doublebar{\nabla_\pureH^m \varphi}_{L^2(\R^n)}^2
.\end{equation}
We will establish well posedness of the Neumann problem only for coefficients that satisfy the stronger ellipticity condition
\begin{equation}
\label{C:eqn:elliptic:slices:strong}
\re\langle \nabla^m \varphi(\,\cdot\,,t),\mat A\nabla^m\varphi(\,\cdot\,,t)\rangle_{\R^n} 
\geq 
	\lambda\doublebar{
	\nabla^m\varphi(\,\cdot\,,t)}_{L^2(\R^n)}^2
\end{equation}
for all $\varphi$ smooth and compactly supported in~$\R^\dmn$ and all $t\in\R$. 

\begin{rmk}\label{C:rmk:biharmonic}
For many applications in the theory, the ellipticity condition \eqref{C:eqn:elliptic} suffices. See, for example, the construction of solutions in $\dot W^2_m(\Omega)$ to the Dirichlet and Neumann problems via the Lax-Milgram theorem, the solution to the Kato square root problem in \cite{AusHMT01}, the well posedness of the $L^2$ and $\dot W^2_1$-Dirichlet problems in \cite{PipV95B}, the boundedness of layer potentials in \cite{BarHM15p,BarHM17pA} (see Section~\ref{C:sec:potentials:bounds}) and the trace theorems of \cite{BarHM17pB} (see Section~\ref{C:sec:trace}).

However, the ellipticity condition \eqref{C:eqn:elliptic} does not suffice to yield well posedness of the Neumann problem even for very nice operators.

As a simple example, let $L$ denote the biharmonic operator $\Delta^2$, and observe that we may associate $L$ to any member $\mat A_\rho$ of a family of real symmetric coefficient matrices; specifically, if $\rho\in\R$, then let $\mat A_\rho$ be such that
\begin{equation*}\langle \nabla^2 \psi(X), \mat A_\rho\nabla^2\varphi(X) \rangle = \rho\langle \Delta \psi(X), \Delta \varphi(X)\rangle 
+(1-\rho)\sum_{j,k=1}^\dmn\langle \partial_{jk}^2\psi(X), \partial_{jk}^2\varphi(X)\rangle\end{equation*}
for any $X\in\R^\dmn$ and any smooth test functions $\varphi$, $\psi$. In the theory of elasticity (see, for example, \cite{Nad63}), the constant $\rho$ is referred to as the Poisson ratio. 

The bounds \eqref{C:eqn:elliptic} and \eqref{C:eqn:elliptic:slices:weak}  are valid regardless of~$\rho$. Furthermore, the choice of $\rho$ does not affect the form of the Dirichlet problem
\begin{equation*}\Delta^2 u=0 \text{ in $\Omega$},\quad \nabla u=\arr f\text{ on $\partial\Omega$}\end{equation*}
and it is known (see \cite{DahKV86,Ver90}) that the Dirichlet problem is well posed in Lipschitz domains with boundary data in $\dot W\!A^2_{m-1,0}(\partial\Omega)$ or $\dot W\!A^2_{m-1,1}(\partial\Omega)$.

However, the Neumann boundary values of a biharmonic function $u$ do depend on the choice of~$\rho$, and the Neumann problem is not well posed for all choices of~$\rho$. In particular, an elementary argument involving the Fourier transform shows that if $\rho=1$ or $\rho=-3$ then the Neumann problem for the biharmonic operator is ill-posed in the half-space, and in \cite{Ver05} the $L^2$-Neumann problem for the Laplacian was shown to be ill-posed in certain planar Lipschitz domains in $\R^\dmn$ for $\rho<-1$ or $\rho\geq 1$. 

Thus, some ellipticity condition beyond~\eqref{C:eqn:elliptic} must be imposed upon the coefficients~$\mat A$; the bound \eqref{C:eqn:elliptic:slices:strong} is the weakest bound that will allow our proof of the Rellich identity to be valid.

\end{rmk}



\subsection{Function spaces and boundary data}

Let $\Omega\subseteq\R^n$ or $\Omega\subseteq\R^\dmn$ be a measurable set in Euclidean space. We will let $L^p(\Omega)$ denote the usual Lebesgue space with respect to Lebesgue measure with norm given by
\begin{equation*}\doublebar{f}_{L^p(\Omega)}=\biggl(\int_\Omega \abs{f(x)}^p\,dx\biggr)^{1/p}.\end{equation*}

If $\Omega$ is a connected open set and $m\geq 1$ is an integer, then we let the homogeneous Sobolev space $\dot W^p_m(\Omega)$ be the space of equivalence classes of functions $u$ that are locally integrable in~$\Omega$ and have weak derivatives in $\Omega$ of order up to~$m$ in the distributional sense, and whose $m$th gradient $\nabla^m u$ lies in $L^p(\Omega)$. Two functions are equivalent if their difference is a polynomial of order~$m-1$.
We impose the norm 
\begin{equation*}\doublebar{u}_{\dot W^p_m(\Omega)}=\doublebar{\nabla^m u}_{L^p(\Omega)}.\end{equation*}
Then $u$ is equal to a polynomial of order $m-1$ (and thus equivalent to zero) if and only if its $\dot W^p_m(\Omega)$-norm is zero. 
We let $L^p_loc(\Omega)$ and $\dot W^p_{k,loc}(\Omega)$ denote functions that lie in $L^p(U)$ (or whose gradients lie in $L^p(U)$) for any bounded open set $U$ with $\overline U\subsetneq\Omega$.

If $1<p<\infty$, we will let $\dot W^p_{-1}(\R^n)$ be the space of bounded linear operators on $\dot W^{p'}_1(\R^n)$, where $1/p+1/p'=1$. Notice that formally, if $g\in \dot W^p_{-1}(\R^n)$ then $g=\nabla_\pureH\cdot \vec h$ for some $\vec h\in L^p(\R^n)$, and conversely that if $h\in L^p(\R^n)$ then $\nabla_\pureH h \in \dot W^p_{-1}(\R^n)$.

%
%


\subsubsection{Dirichlet boundary data and spaces}

In this paper we will establish well posedness of the Neumann problem, and so we are very interested in the Neumann boundary values of solutions. However, Neumann boundary values are defined by duality with Dirichlet boundary values, and so we will need terminology for those values as well.

If $u$ is defined in $\R^\dmn_+$, we let its Dirichelt boundary values be, loosely, the boundary values of the gradient $\nabla^{m-1} u$. More precisely, we let the Dirichlet boundary values be the array of functions $\Tr_{m-1}u=\Tr_{m-1}^+ u$, indexed by multiindices $\gamma$ with $\abs\gamma=m-1$, and given by
\begin{equation}
\label{C:eqn:Dirichlet}
\begin{pmatrix}\Tr_{m-1}^+  u\end{pmatrix}_{\gamma}
=f \quad\text{if}\quad
\lim_{t\to 0^+} \doublebar{\partial^\gamma u(\,\cdot\,,t)-f}_{L^1(K)}=0
\end{equation}
for all compact sets $K\subset\R^n$. If $u$ is defined in $\R^\dmn_-$, we define $\Tr_{m-1}^- u$ similarly. 
We remark that if $\nabla^m u\in L^1(K\times(0,\varepsilon))$ for any such $K$ and some $\varepsilon>0$, then $\Tr_{m-1}^+u$ exists, and furthermore $\begin{pmatrix}\Tr_{m-1}^+  u\end{pmatrix}_{\gamma}=\Trace \partial^\gamma u$ where $\Trace$ denotes the traditional trace in the sense of Sobolev spaces.

We will be concerned with boundary values in Lebesgue or Sobolev spaces. However, observe that the different components of $\Tr_{m-1}u$ arise as derivatives of a common function, and thus must satisfy certain compatibility conditions. We will define the Whitney spaces of functions that satisfy these compatibility conditions and have certain smoothness properties as follows.
\begin{defn} \label{C:dfn:Whitney}
Let 
\begin{equation*}\mathfrak{D}=\{\Tr_{m-1}\varphi:\varphi\text{ smooth and compactly supported in $\R^\dmn$}\}.\end{equation*}

We let $\dot W\!A^2_{m-1,0}(\R^n)$ be the completion of the set $\mathfrak{D}$ under the $L^2$ norm. 

We let $\dot W\!A^2_{m-1,1}(\R^n)$ be the completion of $\mathfrak{D}$ under the $\dot W^2_1(\R^n)$ norm, that is, under the norm $\doublebar{\arr f}_{\dot W\!A^2_{m-1,1}(\R^n)}=\doublebar{\nabla_\pureH \arr f}_{L^2(\R^n)}$. 

Finally, we let $\dot W\!A^2_{m-1,1/2}(\R^n)$ be the completion of $\mathfrak{D}$ under the norm 
\begin{equation}
\label{C:eqn:Besov:norm}
\doublebar{\arr f}_{\dot W\!A^2_{m-1,1/2}(\R^n)} = \biggl(\sum_{\abs\gamma=m-1}\int_{\R^n} \abs{\widehat {f_\gamma}(\xi)}^2\,\abs{\xi}\,d\xi\biggr)^{1/2}
\end{equation}
where $\widehat f$ denotes the Fourier transform of~$f$.
\end{defn}

We are concerned with the spaces $\dot W\!A^2_{m-1,0}(\R^n)$ and $\dot W\!A^2_{m-1,1}(\R^n)$ because we intend to prove well posedness of the Neumann problem with boundary data in their dual spaces $(\dot W\!A^2_{m-1,0}(\R^n))^*$ and $(\dot W\!A^2_{m-1,1}(\R^n))^*$. We will build on the theory of solutions $u$ to elliptic equations with $u\in \dot W^2_m(\R^\dmn_+)$; the space $\dot W\!A^2_{m-1,1/2}(\R^n)$ is important to that theory, as seen in the following lemma.

\begin{lem}\label{C:lem:Besov} 
If $u\in \dot W^2_m(\R^\dmn_+)$ then $\Tr_{m-1}^+u\in \dot W\!A^2_{m-1,1/2}(\R^n)$, and furthermore
\begin{equation*}\doublebar{\Tr_{m-1}^+u}_{\dot W\!A^2_{m-1,1/2}(\R^n)}\leq C \doublebar{\nabla^m u}_{L^2(\R^\dmn_+)}.\end{equation*}
Conversely, if $\arr f\in \dot W\!A^2_{m-1,1/2}(\R^n)$, then there is some $F\in \dot W^2_m(\R^\dmn_+)$ such that $\Tr_{m-1}^+F=\arr f$ and such that
\begin{equation*}\doublebar{\nabla^m F}_{L^2(\R^\dmn_+)}\leq C \doublebar{\arr f}_{\dot W\!A^2_{m-1,1/2}(\R^n)}.\end{equation*}
\end{lem}

If $\dot W^2_m(\R^\dmn_+)$ and $\dot W\!A^2_{m-1,1/2}(\R^n)$ are replaced by their inhomogeneous counterparts, then this lemma is a special case of \cite{Liz60}. For the homogeneous spaces that we consider, the $m=1$ case of this lemma is a special case of \cite[Section~5]{Jaw77}. The trace result for $m\geq 2$ follows from the trace result for $m=1$;  extensions may easily be constructed using the Fourier transform.

%

\subsubsection{Neumann boundary data}
\label{C:sec:dfn:Neumann}

%

We define Neumann boundary values of a solution $u$ to $Lu=0$ as described in the introduction. That is,
define $\mathcal{E}$ as in formula~\eqref{C:eqn:Neumann:extension}.
We define the Neumann boundary values $\M_{\mat A} u=\M_{\mat A}^+ u$ of $u$ by
\begin{equation}\label{C:eqn:Neumann:E}
\langle \M_{\mat A}^+ u,\Tr_{m-1}^+\varphi\rangle_{\R^n}
=\lim_{\varepsilon \to 0^+}\lim_{T\to\infty} \int_\varepsilon^T\langle \mat A\nabla^m u(\,\cdot\,,t), \nabla^m \mathcal{E}\varphi(\,\cdot\,,t)\rangle_{\R^n}\,dt
.\end{equation}
We define $\M_{\mat A}^- u$ similarly, as an appropriate integral from $-\infty$ to zero. 
Notice that $\M_{\mat A} u$ is an operator on the subspace $\mathfrak{D}$ appearing in Definition~\ref{C:dfn:Whitney}; given certain bounds on~$u$, there exist Neumann trace theorems (see Section~\ref{C:sec:trace}) that allow us to extend $\M_{\mat A} u^\pm$ to an operator on $\dot W\!A^2_{m-1,0}(\R^n)$ or $\dot W\!A^2_{m-1,1}(\R^n)$. 

As mentioned in the introduction, if $v$ is as in the Neumann problem~\eqref{C:eqn:neumann:rough} then the inner product $\langle \mat A\nabla^m v(\,\cdot\,,t), \nabla^m \mathcal{E}\varphi(\,\cdot\,,t)\rangle_{\R^n}$ represents an absolutely convergent integral for each \emph{fixed} $t>0$, and the limit in formula~\eqref{C:eqn:Neumann:E} exists, but the integral \eqref{C:eqn:introduction:Neumann:LHS} with $\varphi=\mathcal{E}\varphi$ might not converge absolutely. See Theorem~\ref{C:thm:Neumann:1}. Thus, the order of integration in formula~\eqref{C:eqn:Neumann:E} is important.



However, for solutions that satisfy stronger bounds, we need not be quite so careful in defining Neumann boundary values.

In particular, suppose that  $u\in \dot W^2_m(\R^\dmn_+)$ and that $Lu=0$ in $\R^\dmn_+$. By the definition~\eqref{C:eqn:L} of~$Lu$, if $\varphi$ is smooth and supported in $\R^\dmn_+$, then $\langle \nabla^m\varphi,\mat A\nabla^m u \rangle_{\R^\dmn_+}=0$. By density of smooth functions and boundedness of the trace map, we have that $\langle \nabla^m\varphi,\mat A\nabla^m u \rangle_{\R^\dmn_+}=0$ for any $\varphi\in \dot W^2_m(\R^\dmn_+)$ with $\Tr_{m-1}^+\varphi=0$. Thus, if $\Psi\in \dot W^2_m(\R^\dmn_+)$, then $\langle \nabla^m\Psi,\mat A\nabla^m u\rangle_{\R^\dmn_+}$ depends only on $\Tr_{m-1}^+\Psi$.
Thus, for solutions $u$ to $Lu=0$ with $u\in\dot W^2_m(\R^\dmn_+)$, we may define the Neumann boundary values $\M_{\mat A}^+ u$ by the formula
\begin{equation}\label{C:eqn:Neumann:W2}
\langle \Tr_{m-1}^+\Psi,\M_{\mat A}^+ u\rangle_{\R^n}
=
\langle \nabla^m\Psi,\mat A\nabla^m u\rangle_{\R^\dmn_+} 
\quad\text{for any $\Psi\in \dot W^2_m(\R^\dmn)$}.
\end{equation}
We define $\M_{\mat A}^- u$ for a solution $u\in \dot W^2_m(\R^\dmn_-)$ similarly.
By 
\cite[Lemma~\ref*{B:lem:Neumann:W2}]{BarHM17pB}, if $u\in \dot W^2_m(\R^\dmn_+)$, then the two formulas \eqref{C:eqn:Neumann:E} and~\eqref{C:eqn:Neumann:W2} for the Neumann boundary values of a solution in $\dot W^2_m(\R^\dmn_+)$ coincide.

Furthermore, by  \cite[Theorem~\ref*{B:thm:Neumann:2}]{BarHM17pB}, if $w$ is a solution in  $\R^\dmn_+$ that satisfies estimates as in problem~\eqref{C:eqn:neumann:regular}, then the integral~\eqref{C:eqn:introduction:Neumann:LHS} with $\varphi=\mathcal{E}\varphi$ does converge absolutely for compactly supported~$\varphi$ (and so the order of integration in formula~\eqref{C:eqn:Neumann:E} is not important), and 
\begin{equation*}\langle \Tr_{m-1}^+\varphi,\M_{\mat A}^+ w\rangle_{\R^n}
=\langle \nabla^m \mathcal{E}\varphi,\mat A\nabla^m w\rangle_{\R^\dmn_+}
=\langle \nabla^m \varphi,\mat A\nabla^m w\rangle_{\R^\dmn_+}
\end{equation*}
for any $\varphi\in C^\infty_0(\R^\dmn)$.
Thus, formula~\eqref{C:eqn:Neumann:W2} is valid for $\Psi$ smooth and compactly supported, albeit not for all $\Psi\in \dot W^2_m(\R^\dmn_+)$.

%


See \cite{BarM16B,BarHM15p} for a much more extensive discussion of higher order Neumann boundary values.

\subsection{Potential operators}\label{C:sec:dfn:potentials}
Two very important tools in the theory of second order elliptic boundary value problems are the double and single layer potentials. These potential operators are also very useful in the higher order theory. 
In this section we define our formulations of higher-order layer potentials; this is the formulation used in \cite{BarHM15p,Bar17p,BarHM17pA,BarHM17pB} and is similar to that used in \cite{Agm57,CohG83,CohG85,Ver05,MitM13B,MitM13A}.

For any $\arr H\in L^2(\R^\dmn)$, by the Lax-Milgram theorem there is a unique function $u\in\dot W^2_m(\R^\dmn)$ that satisfies
\begin{equation}\label{C:eqn:newton}
\langle \nabla^m\varphi, \mat A\nabla^m u\rangle_{\R^\dmn}=\langle \nabla^m\varphi, \arr H\rangle_{\R^\dmn}\end{equation}
for all $\varphi\in \dot W^2_m(\R^\dmn)$.
Let $\Pi^L\arr H=u$.  We refer to $\Pi^L$ as the Newton potential operator for~$L$. See \cite{Bar16} for a further discussion of the operator~$\Pi^L$.

We may define the double and single layer potentials in terms of the Newton potential.
Suppose that $\arr f\in \dot W\!A^2_{m-1,1/2}(\R^n)$.
By Lemma~\ref{C:lem:Besov}, there is some  $F\in \dot W^2_m(\R^\dmn_+)$ that satisfies $\arr f=\Tr_{m-1}^+ F$.
We define the double layer potential of $\arr f$ as
\begin{align}
\label{C:dfn:D:newton}
\D^{\mat A}\arr f &= -\1_+ F + \Pi^L(\1_+ \mat A\nabla^m F)
\end{align}
where $\1_+$ is the characteristic function of the upper half-space $\R^\dmn_+$.
$\D^{\mat A}\arr f$ is well-defined, that is, does not depend on the choice of~$F$; see \cite{BarHM15p,Bar17p}. We remark that by \cite[formula~(2.27)]{BarHM15p} or \cite[formula~(4.9)]{Bar17p}, if $\1_-$ is the characteristic function of the lower half space, then
\begin{align}
\label{C:eqn:D:newton:lower}
\D^{\mat A}\arr f &= \1_- F - \Pi^L(\1_- \mat A\nabla^m F) \quad\text{if }\Tr_{m-1}^- F=\arr f.
\end{align}

Similarly, let $\arr g$ be a bounded operator on $\dot W\!A^2_{m-1,1/2}(\R^n)$. 
There is some $\arr G\in L^2(\R^\dmn)$ such that $\langle \arr G, \nabla^m\varphi\rangle_{\R^\dmn} = \langle \arr g, \Tr_{m-1}\varphi\rangle_{\partial{\R^\dmn_+}}$ for all $\varphi\in \dot W^2_m(\R^\dmn)$; see \cite{BarHM15p}. (We may require $\arr G$ to be supported in $\R^\dmn_+$ or $\R^\dmn_-$.)
We define
\begin{align}
\label{C:dfn:S:newton}
\s^{L}\arr g&=\Pi^L\arr G
.\end{align}
Again, $\s^{L}\arr g$ does not depend on the choice of extension~$\arr G$ of~$\arr g$; see \cite{BarHM15p}. 

It was shown in \cite{BarHM17pA} that the operators $\D^{\mat A}$ and $\s^L$, originally defined on $\dot W\!A^{2}_{m-1,1/2}(\R^n)$ and its dual space, extend by density to operators defined on $\dot W\!A^2_{m-1,0}(\R^n)$ and $\dot W\!A^2_{m-1,1}(\R^n)$ or their respective dual spaces; see Section~\ref{C:sec:potentials:bounds}.

A benefit of these formulations of layer potentials is the easy proof of the Green's formula. By taking $F=u$ and $\arr G=\1_+\mat A\nabla^m u$, and applying the definition~\eqref{C:eqn:Neumann:W2} of Neumann boundary values, we immediately have that
\begin{equation}
\label{C:eqn:green}
\1_+ \nabla^m u=-\nabla^m \D^{\mat A}(\Tr_{m-1}^+  u) + \nabla^m \s^{L}(\M_{\mat A}^+  u) 
\end{equation}
for all $u\in\dot W^2_m(\R^\dmn_+)$ that satisfy $Lu=0$ in $\R^\dmn_+$.

We will also need a Green's formula in the lower half space.
If $Lu=0$ in $\R^\dmn_-$ for some $u\in \dot W^2_m(\R^\dmn_-)$, then by formula~\eqref{C:eqn:D:newton:lower},
\begin{equation}
\label{C:eqn:green:lower}
\1_- \nabla^m u=\nabla^m \D^{\mat A}(\Tr_{m-1}^-  u) + \nabla^m \s^{L}(\M_{\mat A}^-  u) 
.\end{equation}

\section{Known results}
\label{C:sec:known}

To prove our main results, we will need to use a number of known results from the theory of higher order differential equations. We gather these results in this section.

\subsection{Regularity of solutions to elliptic equations}

The first such result we list is the higher order analogue to the Caccioppoli inequality; it was proven in full generality in \cite{Bar16} and some important preliminary versions were established in \cite{Cam80,AusQ00}.
\begin{lem}[The Caccioppoli inequality]\label{C:lem:Caccioppoli}
Suppose that $L$ is a divergence-form elliptic operator associated to coefficients $\mat A$ satisfying the ellipticity conditions \eqref{C:eqn:elliptic} and~\eqref{C:eqn:elliptic:bounded}. Let $ u\in \dot W^2_m(B(X,2r))$ with $L u=0$ in $B(X,2r)$.

Then we have the bound
\begin{equation*}
\int_{B(X,r)} \abs{\nabla^j  u(x,s)}^2\,dx\,ds
\leq \frac{C}{r^2}\int_{B(X,2r)} \abs{\nabla^{j-1}  u(x,s)}^2\,dx\,ds
\end{equation*}
for any $j$ with $1\leq j\leq m$.
\end{lem}

If $\mat A$ is $t$-independent, then solutions to $Lu=0$ have additional regularity.
The following lemma was proven in the case $m=1$ in \cite[Proposition 2.1]{AlfAAHK11} and generalized to the case $m\geq 2$ in \cite[Lemma~3.2]{BarHM15p}. 
\begin{lem}\label{C:lem:slices}
Let $t\in\R$ be a constant, and let $Q\subset\R^n$ be a cube with side-length~$\ell(Q)$. Let $2Q$ be the concentric cube of side-length~$2\ell(Q)$.

If $Lu=0$ in $2Q\times(t-\ell(Q),t+\ell(Q))$, and $L$ is an operator of order~$2m$ associated to $t$-independent coefficients~$A$, then
\begin{equation*}\int_Q \abs{\nabla^j \partial_t^k u(x,t)}^2\,dx \leq \frac{C}{\ell(Q)}
\int_{2Q}\int_{t-\ell(Q)}^{t+\ell(Q)} \abs{\nabla^j \partial_s^k u(x,s)}^2\,ds\,dx\end{equation*}
for any $0\leq j\leq m$ and any integer $k\geq 0$.
\end{lem}

\subsection{Boundedness results for layer potentials}
\label{C:sec:potentials:bounds}

We will need the following bounds on layer potentials.

\begin{thm}\label{C:thm:square}\textup{(\cite[Theorem~1.1]{BarHM15p})}
Suppose that $L$ is an elliptic operator of the form \eqref{C:eqn:divergence} of order~$2m$, associated with coefficients $\mat A$ that are $t$-independent in the sense of formula~\eqref{C:eqn:t-independent} and satisfy the ellipticity conditions \eqref{C:eqn:elliptic} and~\eqref{C:eqn:elliptic:bounded}.

Then the operators $\D^{\mat A}$ and $\s^L$, originally defined on $\dot W\!A^{2}_{m-1,1/2}(\R^n)$ and its dual space, extend by density to operators that satisfy
\begin{align}
\label{C:eqn:S:square}
\int_{\R^n}\int_{-\infty}^\infty \abs{\nabla^m \partial_t\s^{L} \arr g(x,t)}^2\,\abs{t}\,dt\,dx
	& \leq C \doublebar{\arr g}_{L^2(\R^n)}^2
,\\
\label{C:eqn:D:square}
\int_{\R^n}\int_{-\infty}^\infty \abs{\nabla^m \partial_t  \D^{\mat A} \arr f(x,t)}^2\,\abs{t}\,dt\,dx
	& \leq C \doublebar{\arr f}_{\dot W^2_1(\R^\dmnMinusOne)}^2
	= C \doublebar{\nabla_\pureH\arr f}_{L^2(\R^n)}^2
\end{align}
for all $\arr g\in {L^2(\R^n)}$ and all $\arr f\in \dot W\!A^2_{m-1,1}(\R^n)$.
\end{thm}

\begin{thm}%
\label{C:thm:square:rough}%
\textup{(\cite[Theorems \ref*{A:thm:S:square:variant} and \ref*{A:thm:D:square:variant}]{BarHM17pA})}
Let $L$ be as in Theorem~\ref{C:thm:square}.
Then $\D^{\mat A}$ and $\s^L$ extend to operators that satisfy
\begin{align}
\label{C:eqn:S:square:variant}
\int_{\R^n}\int_{-\infty}^\infty \abs{\nabla^m \s^{L} \arr g(x,t)}^2\,\abs{t}\,dt\,dx
	& \leq C \doublebar{\arr g}_{\dot W_{-1}^2(\R^n)}^2
,\\
\label{C:eqn:D:square:rough}
\int_{\R^n}\int_{-\infty}^\infty \abs{\nabla^m \D^{\mat A} \arr f(x,t)}^2\,\abs{t}\,dt\,dx
	& \leq C \doublebar{\arr f}_{L^2(\R^n)}^2
\end{align}
for all $\arr g\in {\dot W_{-1}^2(\R^n)}$ and all $\arr f\in \dot W\!A^2_{m-1,0}(\R^n)$.
\end{thm}

\subsection{Trace theorems}
\label{C:sec:trace}

Let $u$ be a solution to $Lu=0$ in $\R^\dmn_\pm$. We will need estimates on the Dirichlet and Neumann boundary values of~$u$. 
We remark that the following theorems are stated only in the upper half-space $\R^\dmn_+$; however, by considering the change of variables $(x,t)\mapsto (x,-t)$, we may derive the corresponding results in the lower half-space.

\begin{thm}[{\cite[Theorem~\ref*{B:thm:Dirichlet:1}]{BarHM17pB}}]
\label{C:thm:Dirichlet:1}
Let $L$ be as in Theorem~\ref{C:thm:square}. Let $v$ satisfy the bound
\begin{equation*}\int_{\R^n}\int_0^\infty \abs{\nabla^m v(x,t)}^2\,t\,dx\,dt<\infty\end{equation*}
and suppose that $Lv=0$ in $\R^\dmn_+$.

Then there is some function $P$ defined in $\R^\dmn_+$ with $\nabla^m P=0$ (that is, a polynomial of degree at most $m-1$) such that
\begin{align*}
\sup_{t>0} \doublebar{\nabla^{m-1} v(\,\cdot\,,t)-\nabla^{m-1} P}_{L^2(\R^n)}^2
&\leq 
C\int_{\R^n}\int_0^\infty \abs{\nabla^m v(x,t)}^2\,t\,dx\,dt
,
\\
\lim_{t\to \infty} \doublebar{\nabla^{m-1} v(\,\cdot\,,t)-\nabla^{m-1} P}_{L^2(\R^n)}&=0
.\end{align*}
Furthermore, there is some array of functions $\arr f\in L^1_{loc}(\R^n)$ such that 
\begin{equation*}\doublebar{\nabla^{m-1} v(\,\cdot\,,t)- \arr f}_{L^2(\R^n)}\to 0\quad\text{ as $t\to 0^+$},\end{equation*}
and such that
\begin{equation*}
\doublebar{\arr f-\nabla^{m-1} P}_{L^2(\R^n)}^2\leq C\int_{\R^n}\int_0^\infty \abs{\nabla^m v(x,t)}^2\,t\,dx\,dt
.\end{equation*}
\end{thm}

\begin{thm} [{\cite[Theorem~\ref*{B:thm:Dirichlet:2}]{BarHM17pB}}]
\label{C:thm:Dirichlet:2}
Let $L$ be as in Theorem~\ref{C:thm:square}. Let $w\in \dot W^2_{m,loc}(\R^\dmn_+)$  satisfy the bound
\begin{equation*}\int_{\R^n}\int_0^\infty \abs{\nabla^m \partial_t w(x,t)}^2\,t\,dx\,dt<\infty\end{equation*}
and suppose that $Lw=0$ in $\R^\dmn_+$.

Then there is some array $\arr p$ of functions defined on $\R^n$ such that
\begin{align*}
\sup_{t>0} \doublebar{\nabla^{m} w(\,\cdot\,,t)-\arr p}_{L^2(\R^n)}^2
&\leq 
C\int_{\R^n}\int_0^\infty \abs{\nabla^m \partial_t w(x,t)}^2\,t\,dx\,dt
,\\
\lim_{t\to\infty} \doublebar{\nabla^{m} w(\,\cdot\,,t)-\arr p}_{L^2(\R^n)}
&=0
.\end{align*}
Furthermore, there is some array of functions $\arr f\in L^1_{loc}(\R^n)$ such that 
\begin{equation*}\doublebar{\nabla^{m} w(\,\cdot\,,t)- \arr f}_{L^2(\R^n)}\to 0\quad\text{ as $t\to 0^+$},\end{equation*}
and such that
\begin{equation*}
\doublebar{\arr f-\arr p}_{L^2(\R^n)}^2\leq C\int_{\R^n}\int_0^\infty \abs{\nabla^m \partial_t w(x,t)}^2\,t\,dx\,dt
.\end{equation*}

If $\nabla^{m} w(\,\cdot\,,t)\in L^2(\R^n)$ for some $t>0$, then $\arr p=0$. 
\end{thm}

\begin{thm}[{\cite[Theorem~\ref*{B:thm:Neumann:1}]{BarHM17pB}}]
\label{C:thm:Neumann:1}
Let $L$ be as in Theorem~\ref{C:thm:square} and let $v$ be as in Theorem~\ref{C:thm:Dirichlet:1}.

Then for all $\varphi$ smooth and compactly supported, we have that
\begin{equation*}{\langle \mat A\nabla^m v(\,\cdot\,,t),\nabla^m \mathcal{E}\varphi(\,\cdot\,,t)\rangle_{\R^n}}\end{equation*}
represents an absolutely convergent integral for any fixed $t>0$ and is continuous in~$t$.

Furthermore,
\begin{multline*}
\sup_{0<\varepsilon<T}
\abs[bigg]{
\int_\varepsilon^T
\langle \mat A\nabla^m v(\,\cdot\,,t),\nabla^m \mathcal{E}\varphi(\,\cdot\,,t)\rangle_{\R^n}\,dt 
}
\\\leq 
C \doublebar{\nabla_\pureH \Tr_{m-1}^+\varphi}_{L^2(\R^n)}
\biggl(\int_{\R^n}\int_0^\infty \abs{\nabla^m v(x,t)}^2\,t\,dx\,dt\biggr)^{1/2}
\end{multline*}
and the limit
\begin{equation*}\lim_{\varepsilon\to 0^+} \lim_{T\to \infty} \int_\varepsilon^T
\langle \mat A\nabla^m v(\,\cdot\,,t),\nabla^m \mathcal{E}\varphi(\,\cdot\,,t)\rangle_{\R^n}\,dt 
\end{equation*}
exists,
and so we have the bound
\begin{align*}
\abs{\langle\M_{\mat A}^+v,\Tr_{m-1}\varphi\rangle_{\R^\dmn_+}}
&\leq C \doublebar{\nabla_\pureH \Tr_{m-1}^+\varphi}_{L^2(\R^n)}
\biggl(\int_{\R^n}\int_0^\infty \abs{\nabla^m v(x,t)}^2\,t\,dx\,dt\biggr)^{1/2}
.\end{align*}
\end{thm}

\begin{thm}[{\cite[Theorem~\ref*{B:thm:Neumann:2}]{BarHM17pB}}]
\label{C:thm:Neumann:2}
Let $L$ be as in Theorem~\ref{C:thm:square}. Let $w$ be as in Theorem~\ref{C:thm:Dirichlet:2}, and suppose further that $\nabla^{m} w(\,\cdot\,,t)\in L^2(\R^n)$ for some $t>0$ (so that $\arr p=0$).

Then for all $\varphi$ smooth and compactly supported in $\R^\dmn$ we have that
\begin{equation*}
\int_0^\infty \int_{\R^n}\abs{\langle \mat A(x)\nabla^m w(x,t),\nabla^m\mathcal{E}\varphi(x,t)\rangle}\,dx\,dt<\infty \end{equation*}
and that the bound
\begin{equation*}
\abs{\langle \M_{\mat A}^+ w,\Tr_{m-1}\varphi\rangle_{\R^n}}
\leq C \doublebar{\Tr_{m-1}^+\varphi}_{L^2(\R^n)}
\biggl(\int_0^\infty \int_{\R^n} \abs{\nabla^m\partial_t w(x,t)}^2\,t\,dx\,dt\biggr)^{1/2}
\end{equation*}
is valid.
\end{thm}

\section{More on boundary values}
\label{C:sec:boundary}

In this section we will provide one more result for the Neumann boundary values of solutions; we will then combine the bounds on layer potentials (Theorems~\ref{C:thm:square} and~\ref{C:thm:square:rough}) with the trace results of Section~\ref{C:sec:trace} to bound the Dirichlet and Neumann boundary values of layer potentials. We remark in particular that some extra analysis is necessary to dispense with the functions  $\arr p$ of Theorem~\ref{C:thm:Dirichlet:2}. Finally, we will generalize the Green's formulas \eqref{C:eqn:green} and~\eqref{C:eqn:green:lower} from solutions in $\dot W^2_m(\R^\dmn_\pm)$ to solutions that satisfy square-function estimates as in Theorems~\ref{C:thm:well-posed:2} and~\ref{C:thm:well-posed:1}.

\subsection{Limits of Neumann boundary values}
\label{C:sec:neumann:limit}

If $\mat A$ is $t$-independent and $Lu=0$, then $Lu_\sigma=0$ as well, where $u_\sigma(x,t)=u(x,t+\sigma)$. It is often useful to analyze $u$ by analyzing $u_\sigma$ and taking a limit as $\sigma\to 0$; for example, if $u$ satisfies the conditions of Theorem~\ref{C:thm:Dirichlet:1} and $\sigma>0$, then $u_\sigma\in \dot W^2_m(\R^\dmn_+)$ and so the Green's formula~\eqref{C:eqn:green} and the weak formulation of Neumann boundary values~\eqref{C:eqn:Neumann:W2} are valid.

Theorems~\ref{C:thm:Dirichlet:1} and~\ref{C:thm:Dirichlet:2} establish uniform bounds on $\Tr_{m-1}^+ u_\sigma$ and show that $\Tr_{m-1}^+ u_\sigma\to \Tr_{m-1}^+ u$ as $\sigma\to 0^+$. Theorems~\ref{C:thm:Neumann:1} and~\ref{C:thm:Neumann:2}, by contrast, bound $\M_{\mat A}^+ u$ alone. While it is clear that if $u$ satisfies the conditions of Theorem~\ref{C:thm:Neumann:1} or~\ref{C:thm:Neumann:2}, then so does~$u_\sigma$, it is not clear that $\M_{\mat A}^+ u_\sigma \to \M_{\mat A}^+ u$; establishing this limit is the goal of this section.

As in Section~\ref{C:sec:trace}, similar results are valid in the lower half-space.


\begin{lem}\label{C:lem:Neumann:limit:1} 
Let $L$ and $v$ be as in Theorem~\ref{C:thm:Neumann:1} (that is, as in Theorems~\ref{C:thm:square} and~\ref{C:thm:Dirichlet:1}).
Let $v_\sigma(x,t)=v(x,t+\sigma)$.

Then 
$\M_{\mat A}^+ v_\varepsilon\to \M_{\mat A}^+ v$ in $L^2(\R^n)$ as $\varepsilon\to 0^+$, and 
$\M_{\mat A}^+ v_T\to 0$ in $L^2(\R^n)$ as $T\to\infty$.
\end{lem}

\begin{proof}
First,
\begin{equation*}\int_{\R^n}\int_0^\infty \abs{\nabla^m v_T(x,t)}^2t\,dt\,dx
=
\int_{\R^n}\int_T^\infty \abs{\nabla^m v(x,t)}^2(t-T)\,dt\,dx \end{equation*}
which approaches zero as $T\to\infty$, and so by Theorem~\ref{C:thm:Neumann:1}, $\M_{\mat A}^+ v_T\to 0$ in $L^2(\R^n)$ as $T\to\infty$.

We now turn to the limit $\M_{\mat A}^+ v_\varepsilon\to \M_{\mat A}^+ v$. It suffices to show that 
\begin{equation*}\lim_{\varepsilon\to 0^+} 
\int_{\R^n}\int_0^\infty \abs{\nabla^m v_\varepsilon(x,t)-\nabla^m v(x,t)}^2\,t\,dt\,dx= 0
.\end{equation*}
But
\begin{multline*}
\int_0^\infty\int_{\R^n}
	\abs{\nabla^m v_\varepsilon(x,t)-\nabla^m v(x,t)}^2\,t\,dx\,dt
\\\begin{aligned}
&\leq
	2\int_0^{\sqrt{\varepsilon}}\int_{\R^n}
	\abs{\nabla^m v_\varepsilon(x,t)}^2\,t\,dx\,dt
	+
	2\int_0^{\sqrt{\varepsilon}}\int_{\R^n}
	\abs{\nabla^m v(x,t)}^2\,t\,dx\,dt
	\\&\qquad+
	\int_{\sqrt{\varepsilon}}^\infty\int_{\R^n}
	\abs{\nabla^m v_\varepsilon(x,t)-\nabla^m v(x,t)}^2\,t\,dx\,dt
.\end{aligned}\end{multline*}
Recalling the definition of $v_\varepsilon$, the first two integrals on the right-hand side may be bounded by 
\begin{equation*}2\int_\varepsilon^{\varepsilon+\sqrt{\varepsilon}}\int_{\R^n}
	\abs{\nabla^m v(x,t)}^2\,(t-\varepsilon)\,dx\,dt
	+
	2\int_0^{\sqrt{\varepsilon}}\int_{\R^n}
	\abs{\nabla^m v(x,t)}^2\,t\,dx\,dt\end{equation*}
which approaches zero as $\varepsilon\to \infty$.

The final integral is at most
\begin{align*}
	\int_{\sqrt{\varepsilon}}^\infty\int_{\R^n}
	\abs[bigg]{\int_t^{t+\varepsilon}\nabla^m \partial_s v(x,s)\,ds}^2\,t\,dx\,dt
&\leq
	\varepsilon\int_{\sqrt{\varepsilon}}^\infty\int_{\R^n}
	\int_t^{t+\varepsilon}\abs{\nabla^m \partial_s v(x,s)}^2\,ds\,t\,dx\,dt
\\&\leq
	\varepsilon^2\int_{\R^n}\int_{\sqrt{\varepsilon}}^\infty
	\abs{\nabla^m \partial_s v(x,s)}^2\,s\,ds\,dx
.\end{align*}
Applying the Caccioppoli inequality in cubes of side-length $\sqrt{\varepsilon}/C$, we see that
\begin{align*}
	\int_{\sqrt{\varepsilon}}^\infty\int_{\R^n}
	\abs[bigg]{\int_t^{t+\varepsilon}\nabla^m \partial_s v(x,s)\,ds}^2\,t\,dx\,dt
&\leq
	C\varepsilon\int_{\R^n}\int_{\sqrt{\varepsilon}/2}^\infty
	\abs{\nabla^m v(x,s)}^2\,s\,ds\,dx
\end{align*}
which, again, approaches zero as $\varepsilon\to 0^+$.
\end{proof}

\begin{lem}\label{C:lem:Neumann:limit:2} 
Let $L$ and $w$ be as in Theorem~\ref{C:thm:Neumann:2}. 
Let $w_\sigma(x,t)=w(x,t+\sigma)$.

Then 
$\M_{\mat A}^+ w_\varepsilon\to \M_{\mat A}^+ w$ as $\varepsilon\to 0^+$ in $L^2(\R^n)$, and $M_{\mat A}^+ w_T\to 0$ as $T\to \infty$.
\end{lem}

\begin{proof} Clearly $\nabla^m w_\sigma(\,\cdot\,,t)\in L^2(\R^n)$ for every $t>0$.
Arguing as in the proof of Lemma~\ref{C:lem:Neumann:limit:1}, we have that
\begin{gather*}\lim_{\varepsilon\to 0^+} 
\int_{\R^n}\int_0^\infty \abs{\nabla^m \partial_t w_\varepsilon(x,t)-\nabla^m \partial_t w(x,t)}^2\,t\,dt\,dx= 0,
\\
\lim_{T\to\infty} \int_{\R^n}\int_0^\infty \abs{\nabla^m \partial_t w_T(x,t)}^2\,t\,dt\,dx=0
.\end{gather*}
and by Theorem~\ref{C:thm:Neumann:2} the proof is complete.
\end{proof}

\subsection{Dirichlet boundary values of layer potentials}
\label{C:sec:Dirichlet:layer}

Recall the bounds on layer potentials of Section~\ref{C:sec:potentials:bounds}.
By Theorem~\ref{C:thm:square:rough}, we have that if $\arr g\in {\dot W_{-1}^2(\R^n)}$ and if $\arr f\in \dot W\!A^2_{m-1,0}(\R^n)\subset L^2(\R^n)$, then $v=\s^L\arr g$ or $v=\D^{\mat A}\arr f$ satisfies the conditions of Theorem~\ref{C:thm:Dirichlet:1}; thus, there exist polynomials $P_g$ and $P_f$ such that
\begin{align*}
\sup_{t\neq 0} \doublebar{\nabla^{m-1} \s^{L} \arr g(\,\cdot\,,t)-\nabla^{m-1}P_g}_{L^2(\R^n)} &\leq C \doublebar{\arr g}_{\dot W_{-1}^2(\R^n)}
,\\
\sup_{t\neq 0} \doublebar{\nabla^{m-1} \D^{\mat A} \arr f(\,\cdot\,,t)-\nabla^{m-1}P_f}_{L^2(\R^n)} &\leq C \doublebar{\arr f}_{L^2(\R^n)}
= C\doublebar{\arr f}_{\dot W\!A^2_{m-1,0}(\R^n)}
.\end{align*}
Recall that $\D^{\mat A}$ and $\s^L$ were originally defined as operators from $\dot W\!A^2_{m-1,1/2}(\R^n)$ and its dual space to $\dot W^2_m(\R^\dmn_\pm)$, a space defined modulo polynomials. By Theorem~\ref{C:thm:square:rough}, we may extend $\D^{\mat A}$ and $\s^L$ to operators on $\dot W\!A^2_{m-1,0}(\R^n)$ and $\dot W^2_{-1}(\R^n)$; however, we again have that $\D^{\mat A} \arr f$ and $\s^L\arr g$ are only locally Sobolev functions, that is, are defined only up to adding polynomials of degree $m-1$. We adopt the convention that the polynomials $P_g$ and $P_f$ are of degree $m-2$; that is, we normalize $u=\D^{\mat A} \arr f$ and $u=\s^L\arr g$ so that $\nabla^{m-1}u(\,\cdot\,,t)\to 0$ as $t\to\infty$.
Thus, we have the bounds
\begin{align}
\label{C:eqn:S:L2}
\sup_{t\neq 0} \doublebar{\nabla^{m-1} \s^{L} \arr g(\,\cdot\,,t)}_{L^2(\R^n)} &\leq C \doublebar{\arr g}_{\dot W_{-1}^2(\R^n)}
,\\
\label{C:eqn:D:L2}
\sup_{t\neq 0} \doublebar{\nabla^{m-1} \D^{\mat A} \arr f(\,\cdot\,,t)}_{L^2(\R^n)} &\leq C \doublebar{\arr f}_{L^2(\R^n)}
= C\doublebar{\arr f}_{\dot W\!A^2_{m-1,0}(\R^n)}
.\end{align}
Furthermore, for such $\arr f$ and $\arr g$, $\nabla^{m-1} \s^{L} \arr g(\,\cdot\,,t)$ and ${\nabla^{m-1} \D^{\mat A} \arr f(\,\cdot\,,t)}$ approach zero in $L^2(\R^n)$ as $t\to\pm \infty$, and approach (usually nonzero) limits in $L^2(\R^n)$ as $t\to 0^\pm$; that is, $\Tr_{m-1}^\pm \D^{\mat A}$ and $\Tr_{m-1}^\pm\s^L$ are bounded operators $\dot W\!A^2_{m-1,0}(\R^n)\mapsto \dot W\!A^2_{m-1,0}(\R^n)$ and $\dot W^2_{-1}(\R^n)\mapsto \dot W\!A^2_{m-1,0}(\R^n)$.

We remark that if $\D^{\mat A}$ and $\s^L$ are defined using the fundamental solution, as in \cite{BarHM17pA}, then this naturalization condition follows from the normalization conditions of the fundamental solution; see \cite[Remark~\ref*{A:rmk:potentials:decay}]{BarHM17pA}.

We now turn to the bounds given by Theorem~\ref{C:thm:square} rather than Theorem~\ref{C:thm:square:rough}.
By Theorem~\ref{C:thm:square}, we have that if $\arr g\in L^2(\R^n)$ and if $\arr f\in \dot W\!A^2_{m-1,1}(\R^n)\subset \dot W^2_1(\R^n)$, then $w=\s^L\arr g$ or $w=\D^{\mat A}\arr f$ satisfies the conditions of Theorem~\ref{C:thm:Dirichlet:2}; thus, there exist constants $\arr p_g$ and $\arr p_f$ such that
\begin{align*}
\sup_{t\neq 0} \doublebar{\nabla^{m} \s^{L} \arr g(\,\cdot\,,t)-\arr p_g}_{L^2(\R^n)} &\leq C \doublebar{\arr g}_{L^2(\R^n)}
,\\
\sup_{t\neq 0} \doublebar{\nabla^{m} \D^{\mat A} \arr f(\,\cdot\,,t)-\arr p_f}_{L^2(\R^n)} &\leq C \doublebar{\arr f}_{\dot W^2_1(\R^n)}
= C\doublebar{\arr f}_{\dot W\!A^2_{m-1,1}(\R^n)}.
.\end{align*}
But recall that $\D^{\mat A}$ is bounded ${\dot W\!A^2_{m-1,1/2}(\R^n)}\mapsto \dot W^2_m(\R^\dmn_+)$, and $\s^L$ is bounded $({\dot W\!A^2_{m-1,1/2}(\R^n)})^*\mapsto \dot W^2_m(\R^\dmn_+)$. If $Lw=0$ for some $w\in \dot W^2_m(\R^\dmn_+)$, then 
by Lemma~\ref{C:lem:slices} applied in cubes of side-length $t/2$ we have that $\doublebar{\nabla^m w(\,\cdot\,,t)}_{L^2(\R^n)}^2 \leq (C/t)\doublebar{\nabla^m w}_{L^2(\R^\dmn_+)}^2$. In particular $\doublebar{\nabla^m w(\,\cdot\,,t)}_{L^2(\R^n)}$ is finite for all $t>0$. Thus, $\arr p_f=0$ and $\arr p_g=0$ for $\arr f\in {\dot W\!A^2_{m-1,1/2}(\R^n)}\cap {\dot W\!A^2_{m-1,1}(\R^n)}$ and~$\arr g\in ({\dot W\!A^2_{m-1,1/2}(\R^n)})^*\cap L^2(\R^n)$.

By density, for all $\arr g\in L^2(\R^n)$, $\arr f\in \dot W\!A^2_{m-1,1}(\R^n)$, we have that
\begin{align}
\label{C:eqn:S:L2:smooth}
\sup_{t\neq 0} \doublebar{\nabla^m \s^{L}\arr g(\,\cdot\,,t)}_{L^2(\R^n)} &\leq C\doublebar{\arr g}_{L^2(\R^n)}
,\\
\label{C:eqn:D:L2:smooth}
\sup_{t\neq 0} \doublebar{\nabla^m \D^{\mat A}\arr f(\,\cdot\,,t)}_{L^2(\R^n)} &\leq C\doublebar{\nabla_\pureH\arr f}_{L^2(\R^n)}=C\doublebar{\arr f}_{\dot W\!A^2_{m-1,1}(\R^n)}
.\end{align}
Furthermore, for such $\arr f$ and $\arr g$, $\nabla^{m} \s^{L} \arr g(\,\cdot\,,t)$ and ${\nabla^{m} \D^{\mat A} \arr f(\,\cdot\,,t)}$ approach zero in $L^2(\R^n)$ as $t\to\pm \infty$, and approach (usually nonzero) limits in $L^2(\R^n)$ as $t\to 0^\pm$; that is, $\Tr_{m-1}^\pm \D^{\mat A}$ and $\Tr_{m-1}^\pm\s^L$ are bounded operators $\dot W\!A^2_{m-1,1}(\R^n)\mapsto \dot W\!A^2_{m-1,1}(\R^n)$ and $L^2(\R^n)\mapsto \dot W\!A^2_{m-1,1}(\R^n)$.

\subsection{Neumann boundary values of layer potentials}\label{C:sec:Neumann:layer}
The case of Neumann boundary values is somewhat simpler.
By Theorems~\ref{C:thm:square:rough} and \ref{C:thm:Neumann:1}, we have that 
\begin{align}
\label{C:eqn:Neumann:S:1}
\doublebar{\M_{\mat A} \s^{L}\arr g}_{(\dot W\!A^2_{m-1,1}(\R^n))^*}
&\leq C\doublebar{\arr g}_{\dot W_{-1}^2(\R^n)}
,\\
\label{C:eqn:Neumann:D:1}
\doublebar{\M_{\mat A} \D^{\mat A}\arr f}_{(\dot W\!A^2_{m-1,1}(\R^n))^*}
&\leq C\doublebar{\arr f}_{\dot W\!A^2_{m-1,0}(\R^n)}
\end{align}
for any $\arr f\in \dot W\!A^2_{m-1,0}(\R^n)$ and any $\arr g\in \dot W_{-1}^2(\R^n)$.

Furthermore, by Theorems~\ref{C:thm:square} and~\ref{C:thm:Neumann:2}, and the bounds \eqref{C:eqn:S:L2:smooth} and~\eqref{C:eqn:D:L2:smooth}, we have that 
\begin{align}
\label{C:eqn:Neumann:S:2}
\doublebar{\M_{\mat A} \s^{L}\arr g}_{(\dot W\!A^2_{m-1,0}(\R^n))^*}
&\leq C\doublebar{\arr g}_{L^2(\R^n)}
,\\
\label{C:eqn:Neumann:D:2}
\doublebar{\M_{\mat A} \D^{\mat A}\arr f}_{(\dot W\!A^2_{m-1,0}(\R^n))^*}
&\leq C\doublebar{\arr f}_{\dot W\!A^2_{m-1,1}(\R^n)}
\end{align}
for any $\arr f\in \dot W\!A^2_{m-1,0}(\R^n)$ and any $\arr g\in L^2(\R^n)$.

\subsection{The Green's formula}
\label{C:sec:green}

Recall that if $Lu=0$ in $\R^\dmn_\pm$ and $u\in \dot W^2_m(\R^\dmn_\pm)$, then $u$ satisfies the Green's formula \eqref{C:eqn:green} or~\eqref{C:eqn:green:lower}.
We are chiefly concerned with solutions that satisfy square-function estimates, as in Section~\ref{C:sec:trace}; thus, we would like to show that such functions satisfy the Green's formula as well.

\begin{thm}\label{C:thm:green}
Let $L$ be as in Theorem~\ref{C:thm:square}.

Let $v$ satisfy the conditions of Theorem~\ref{C:thm:Dirichlet:1} or the corresponding condition in the lower half-space. Then the Green's formula \eqref{C:eqn:green} or~\eqref{C:eqn:green:lower} is valid for $u=v$.

Similarly, 
let $w$ satisfy the conditions of Theorem~\ref{C:thm:Neumann:2} or the corresponding condition in the lower half-space. Then the Green's formula \eqref{C:eqn:green} or~\eqref{C:eqn:green:lower} is valid for $u=w$.
\end{thm}

\begin{proof} We will work only in the upper half-space $\R^\dmn_+$; the argument in $\R^\dmn_-$ is similar.


Let $w_\varepsilon(x,t)= w(x,t+\varepsilon)$,
and let $w_{\varepsilon,T} = w_\varepsilon-w_T$. 
Then $\partial_\dmn w_\tau=\partial_\tau w_\tau\in \dot W^2_m(\R^\dmn_+)$ for any $\tau>0$; because \begin{equation*}w_{\varepsilon,T}=-\int_\varepsilon^T \partial_\tau w_\tau\,d\tau,\end{equation*}
we have that $w_{\varepsilon,T}\in \dot W^2_m(\R^\dmn_+)$ for any $0<\varepsilon<T$.
Thus, by formula~\eqref{C:eqn:green},
\begin{multline*}\nabla^m w(x,t+\varepsilon)-\nabla^m w(x,t+T) 
\\= -\nabla^m \D^{\mat A}(\Tr_{m-1}^+ w_{\varepsilon,T})(x,t) + \nabla^m \s^{L} (\M_{\mat A}^+ w_{\varepsilon,T})(x,t).\end{multline*}
We take the limit of all four terms as $\varepsilon\to 0^+$ and as $T\to\infty$.

By Theorem~\ref{C:thm:Dirichlet:2}, we have that $\nabla^m w(\,\cdot\,,t+T)\to 0$ in $L^2(\R^n)$ as $T\to \infty$; by Theorem~\ref{C:thm:Dirichlet:2}, Lemma~\ref{C:lem:slices} and the Caccioppoli inequality, $\nabla^m w(\,\cdot\,,t+\varepsilon)\to \nabla^m w(\,\cdot\,,t)$ in $L^2(\R^n)$ as $\varepsilon\to 0^+$.

By Theorem~\ref{C:thm:Dirichlet:2}, the above limits are valid for $t=0$; thus,  $\Tr_{m-1}^+ w_{\varepsilon,T}\to \Tr_{m-1}^+ w$ in $\dot W\!A^2_{m-1,1}(\R^n)$ as $\varepsilon\to 0^+$ and $T\to\infty$. By boundedness of the double layer potential (the bound~\eqref{C:eqn:D:square}) and by Theorem~\ref{C:thm:Dirichlet:2}, the Caccioppoli inequality and Lemma~\ref{C:lem:slices}, $\nabla^m\D^{\mat A}(\Tr_{m-1}^+ w_{\varepsilon,T})(\,\cdot\,,t)\to \nabla^m\D^{\mat A}(\Tr_{m-1}^+ w)(\,\cdot\,,t)$ in $L^2(\R^n)$.

Finally, by Lemma~\ref{C:lem:Neumann:limit:2}, $\M_{\mat A}^+ w_{\varepsilon}\to \M_{\mat A}^+ w$ and $\M_{\mat A}^+ w_T\to 0$ in $L^2(\R^n)$ as $\varepsilon\to 0^+$ and $T\to \infty$. By boundedness of the single layer potential (the bound~\eqref{C:eqn:S:square}) and by the Caccioppoli inequality and Theorem~\ref{C:thm:Dirichlet:2}) and Lemma~\ref{C:lem:slices}, we have that $\nabla^m\s^{L}(\M_{\mat A}^+  w_{\varepsilon,T})(\,\cdot\,,t)\to \nabla^m\s^{L}(\M_{\mat A}^+  w)(\,\cdot\,,t)$ in $L^2(\R^n)$.

Thus, the Green's formula is valid.

The same argument is valid for~$v$; in fact, $v_\sigma\in \dot W^2_m(\R^\dmn_+)$ for any $\sigma>0$, and so we may work with $v_\varepsilon$ and not $v_{\varepsilon,T}$.
\end{proof}

\section{The Rellich identity and uniqueness of solutions}
\label{C:sec:rellich}

The second-order Rellich identity is one of the cornerstones of the theory. In the following theorem we provide a one-sided higher order generalization. This generalization is enough to prove uniqueness of solutions to the Neumann problem~\eqref{C:eqn:neumann:regular}.

\begin{thm}\label{C:thm:rellich} Suppose that $L$ is an elliptic operator of order $2m$ associated with coefficients $\mat A$ that are $t$-independent in the sense of formula~\eqref{C:eqn:t-independent} and satisfy the ellipticity conditions \eqref{C:eqn:elliptic:slices:strong} and~\eqref{C:eqn:elliptic:bounded}.

Suppose in addition that the coefficients $\mat A$ are self-adjoint; that is, that $A_{\alpha\beta}=\overline{A_{\beta\alpha}}$ for any $\abs\alpha=\abs\beta=m$.

Let $w$ satisfy the conditions of Theorems~\ref{C:thm:Dirichlet:2} and~\ref{C:thm:Neumann:2}. That is, suppose that $Lw=0$ in $\R^\dmn_+$, that $\int_0^\infty \int_{\R^n} \abs{\nabla^m \partial_t w(x,t)}^2,t,dx,dt<\infty$, and that $\nabla^m w(\,\cdot\,,t)\in L^2(\R^n)$ for some (hence every) $t>0$. 

By Theorems~\ref{C:thm:Dirichlet:2} and~\ref{C:thm:Neumann:2}, $\Tr_m w$ exists as an $L^2(\R^n)$ function, and ${\M_{\mat A}^+ w}$ exists as a linear operator on $\dot W^2_{m-1,0}(\R^n)$.
Then we have the bound
\begin{align*}
\int_{\R^n} \abs{\Tr_m w(x)}^2\,dx
&\leq 
	-\frac{2}{\lambda}  \re \langle \Tr_{m-1}\partial_\dmn w, \M_{\mat A}^+ w\rangle_{\R^n}
\end{align*}
and so
\begin{equation*}\doublebar{\Tr_{m-1}^+ w}_{\dot W^2_1(\R^n)}
\leq \doublebar{\Tr_m^+ w}_{L^2(\R^n)}
\leq
C \doublebar{\M_{\mat A}^+ w}_{(\dot W\!A^2_{m-1,0}(\R^n))^*}
.\end{equation*}
\end{thm}

Because $\dot W\!A^2_{m-1,0}(\R^n)$ is a closed subset of $L^2(\R^n)$, we may extend any linear operator on $\dot W\!A^2_{m-1,0}(\R^n)$ to a linear operator on $L^2(\R^n)$, that is, to an $L^2$ function; thus, we have the bound
\begin{equation*}\doublebar{\Tr_m w}_{L^2(\R^n)}
\leq
C \doublebar{\M_{\mat A}^+ w}_{L^2(\R^n)}
.\end{equation*}

\begin{proof}[Proof of Theorem~\ref{C:thm:rellich}]
First, observe that, for any $t>0$, by the bound \eqref{C:eqn:elliptic:slices:strong},
\begin{equation*}\int_{\R^n} \abs{\nabla^m w(x,t)}^2\,dx
\leq \frac{1}{\lambda}  \int_{\R^n} \langle \nabla^m w(x,t),\mat A\nabla^m w(x,t)\rangle\,dx.\end{equation*}
Because $\mat A$ is self-adjoint, the integrand is necessarily real-valued. 

Let $w_{\sigma}(x,t) = w(x,t+\sigma)$ and let $w_{\varepsilon,T}=w_\varepsilon-w_T$. For any $\sigma>0$ we have that $\partial_\dmn w_\sigma=\partial_\sigma w_\sigma \in \dot W^2_m(\R^\dmn_+)$. Integrating $\partial_\sigma w_\sigma$ from $\sigma=\varepsilon$ to $\sigma=T$, as in the proof of Lemma~\ref{C:lem:invertible:2}, we have that $w_{\varepsilon,T}\in \dot W^2_m(\R^\dmn_+)$.

Now,
\begin{align*}
\int_{\R^n} \abs{\nabla^m w_{\varepsilon,T}(x,0)}^2\,dx
&\leq
 \frac{1}{\lambda}  \int_{\R^n} \langle \nabla^m w_{\varepsilon,T}(x,0),\mat A\nabla^m w_{\varepsilon,T}(x,0)\rangle\,dx
.\end{align*}
By Theorem~\ref{C:thm:Dirichlet:2}, we have that $\lim_{t\to \infty} \nabla^m w(\,\cdot\,,t)\to 0 $ in $L^2(\R^n)$, and so
\begin{align*}
\int_{\R^n} \abs{\nabla^m  w_{\varepsilon,T}(x,0)}^2\,dx
&\leq
 -\frac{1}{\lambda}  \int_0^\infty\frac{d}{dt}
\int_{\R^n} \langle \nabla^m w_{\varepsilon,T}(x,t),\mat A\nabla^m w_{\varepsilon,T}(x,t)\rangle\,dx\,dt
.\end{align*}
Because $\mat A$ is $t$-independent, we have that 
\begin{align*}
\frac{d}{dt}\langle \nabla^m w_{\varepsilon,T}(x,t),\mat A\nabla^m w_{\varepsilon,T}(x,t)\rangle
&=\langle \nabla^m \partial_t w_{\varepsilon,T}(x,t),\mat A\nabla^m w_{\varepsilon,T}(x,t)\rangle\\&\qquad+\langle \nabla^m w_{\varepsilon,T}(x,t),\mat A\nabla^m \partial_t w_{\varepsilon,T}(x,t)\rangle,\end{align*}
and again because $\mat A$ is self-adjoint, we have that
\begin{equation*}
\int_{\R^n} \abs{\nabla^m  w_{\varepsilon,T}(x,0)}^2\,dx
\leq -\frac{2}{\lambda} \re   \int_0^\infty  \int_{\R^n} \langle \nabla^m \partial_t w_{\varepsilon,T}(x,t),\mat A\nabla^m w_{\varepsilon,T}(x,t)\rangle\,dx\,dt
.\end{equation*}
Recall $w_{\varepsilon,T}\in \dot W^2_m(\R^\dmn_+)$ and $\partial_\dmn w_{\varepsilon,T}\in \dot W^2_m(\R^\dmn_+)$. Thus, by formula~\eqref{C:eqn:Neumann:W2} for the Neumann boundary values of a $\dot W^2_m(\R^\dmn_+)$-function, we have that
\begin{align*}
\int_{\R^n} \abs{\nabla^m  w_{\varepsilon,T}(x,0)}^2\,dx
&\leq -\frac{2}{\lambda} \re   \langle \Tr_{m-1}^+\partial_\dmn w_{\varepsilon,T},\M_{\mat A}^+ w_{\varepsilon,T}\rangle_{\R^n}
.\end{align*}
Now, because the definitions \eqref{C:eqn:Neumann:E} and \eqref{C:eqn:Neumann:W2} of Neumann boundary values coincide for $\dot W^2_m(\R^\dmn_+)$-functions, we have that $\M_{\mat A}^+ w_{\varepsilon,T}=\M_{\mat A}^+ w_{\varepsilon}-\M_{\mat A}^+ w_{T}$ where the two terms on the right-hand side are given by formula~\eqref{C:eqn:Neumann:E} and extend to bounded operators on $\dot W\!A^2_{m-1,0}(\R^n)$.

Thus, we have that
\begin{multline*}
\doublebar{\Tr_m^+ w_\varepsilon-\Tr_m^+ w_T}_{L^2(\R^n)}
\\\leq -\frac{2}{\lambda} \re   \langle \Tr_{m-1}^+\partial_\dmn w_{\varepsilon}-\Tr_{m-1}^+\partial_\dmn w_T,\M_{\mat A}^+ w_{\varepsilon}-\M_{\mat A}^+ w_T\rangle_{\R^n}
.\end{multline*}
Expanding the inner products, we see that
\begin{multline*}
\doublebar{\Tr_{m}^+ w_\varepsilon}_{L^2(\R^n)}^2 
+\doublebar{\Tr_{m}^+ w_T}_{L^2(\R^n)}^2
-2\doublebar{\Tr_{m}^+ w_\varepsilon}_{L^2(\R^n)} 
  \doublebar{\Tr_{m}^+ w_T}_{L^2(\R^n)}
\\\leq 
-\frac{2}{\lambda} \re   \langle \Tr_{m-1}^+\partial_\dmn w_{\varepsilon},\M_{\mat A}^+ w_{\varepsilon}\rangle_{\R^n}
+\frac{2}{\lambda} \re   \langle \Tr_{m-1}^+\partial_\dmn w_{\varepsilon},\M_{\mat A}^+ w_T\rangle_{\R^n}
\\
+\frac{2}{\lambda} \re   \langle \Tr_{m-1}^+\partial_\dmn w_T,\M_{\mat A}^+ w_{\varepsilon}\rangle_{\R^n}
-\frac{2}{\lambda} \re   \langle \Tr_{m-1}^+\partial_\dmn w_T,\M_{\mat A}^+ w_T\rangle_{\R^n}
.\end{multline*}
By Theorem~\ref{C:thm:Dirichlet:2}, $\nabla^m w_\sigma(\,\cdot\,,0)$ is bounded in $L^2(\R^n)$, uniformly in~$\sigma$.
By Theorem~\ref{C:thm:Neumann:2}, the same is true of $\M_{\mat A}^+ w_\sigma$. Again by Theorem~\ref{C:thm:Dirichlet:2}, $\Tr_{m-1}^+ w_\varepsilon\to \Tr_{m-1}^+ w$ and $\Tr_{m-1}^+ w_T\to 0$ in $\dot W_1^2(\R^n)$ as $\varepsilon\to 0^+$ and $T\to \infty$. By Lemma~\ref{C:lem:Neumann:limit:2}, $\M_{\mat A}^+ w_\varepsilon\to \M_{\mat A}^+ w$ in $L^2(\R^n)$ as $\varepsilon\to 0^+$ and $\M_{\mat A}^+ w_T\to 0$ in $L^2(\R^n)$ as $T\to\infty$.

Thus, taking appropriate limits, we have that
\begin{equation*}
\doublebar{\Tr_m^+ w}_{L^2(\R^n)}^2
\leq 
-\frac{2}{\lambda} \re   \langle \Tr_{m-1}^+\partial_\dmn w,\M_{\mat A}^+ w\rangle_{\R^n}
\end{equation*}
as desired.
\end{proof}

We now use the Rellich identity to establish uniqueness of solutions to the $L^2$-Neumann problem~\eqref{C:eqn:neumann:regular}.

\begin{thm}\label{C:thm:neumann:unique}
Let $\mat A$ and $w$ satisfy the conditions of Theorem~\ref{C:thm:rellich}. 
Then 
\begin{equation*}
\sup_{t>0}\doublebar{\nabla^m w(\,\cdot\,,t)}_{L^2(\R^n)}^2+ \int_0^\infty \int_{\R^n} \abs{\nabla^m \partial_t w(x,t)}^2\,t\,dx\,dt
\leq C\doublebar{\M_{\mat A}^+ w}_{L^2(\R^n)}^2
.\end{equation*}
In particular, if $\M_{\mat A}^+ w=0$ then $\nabla^m w\equiv 0$ in $\R^\dmn_+$.
\end{thm}

\begin{proof}
By Theorem~\ref{C:thm:Dirichlet:2}, 
\begin{equation*}\sup_{t>0}\doublebar{\nabla^m w(\,\cdot\,,t)}_{L^2(\R^n)}^2\leq C \int_0^\infty \int_{\R^n} \abs{\nabla^m \partial_t w(x,t)}^2\,t\,dx\,dt.\end{equation*}
By Theorem~\ref{C:thm:green}, we have that $\nabla^m w=-\nabla^m \D^{\mat A} (\Tr_{m-1}^+ w) +\nabla^m \s^L(\M_{\mat A}^+ w)$. Thus, by Theorem~\ref{C:thm:square}, we have that
\begin{equation}\label{C:eqn:Green:control}
\int_0^\infty \int_{\R^n} \abs{\nabla^m \partial_t w(x,t)}^2,t,dx,dt \leq C\doublebar{\Tr_{m-1}^+ w}_{\dot W_1^2(\R^n)}^2+ C\doublebar{\M_{\mat A}^+ w}_{L^2(\R^n)}^2.\end{equation}
By Theorem~\ref{C:thm:rellich}, $\doublebar{\Tr_{m-1}^+ w}_{\dot W_1^2(\R^n)}\leq C\doublebar{\M_{\mat A}^+ w}_{L^2(\R^n)}$
and the proof is complete.
\end{proof}

\begin{rmk}\label{C:rmk:Dirichlet:good}
As mentioned in the introduction, contrary to the present case, it is often easier to solve the Dirichlet or Dirichlet regularity problem than the Neumann problem, and indeed it is often easier to formulate the Dirichlet problem than the Neumann problem.

However, observe that the bound~\eqref{C:eqn:Green:control} is essentially control on solutions in terms of the Dirichlet and Neumann boundary values. Thus, to derive uniqueness of solutions to the Dirichlet problem using this bound, we must bound the Neumann boundary values, and vice versa. 

Theorem~\ref{C:thm:rellich} allows us to control the Dirichlet boundary values by the Neumann boundary values. In a sense, we may say that the Dirichlet boundary values of a solution are at least as well behaved as the Neumann boundary values. Thus, it is still the case in the present context that Dirichlet boundary values are better behaved and easier to work with than Neumann boundary values; the arguments based on the Green's formula mean that good behavior of the Dirichlet boundary values implies uniqueness for the Neumann problem, not the Dirichlet problem.
\end{rmk}

Thus, to establish well posedness of the problem~\eqref{C:eqn:neumann:regular}, it suffices to establish only that solutions exist.

We remark that as usual, a corresponding result is valid in the lower half-space.

\section{Existence of solutions in a special case}
\label{C:sec:special}

In this section, we will prove the following theorem.

\begin{thm}\label{C:thm:special} For any given $n$ and $m$, there is an operator $L$ of order $2m$, acting on functions defined on~$\R^\dmn$, and associated to real constant coefficients $\mat A$ that satisfy the bound \eqref{C:eqn:elliptic:bounded}, the ellipticity condition \eqref{C:eqn:elliptic:slices:strong}, and are self-adjoint, such that the Neumann problem \eqref{C:eqn:neumann:regular} is well posed.
\end{thm}

As discussed above, we need only show that solutions exist. 

Recall that we are working in a very nice domain (the upper half-space). In the case of the Dirichlet (or regularity) problem, the theorem is straightforward to prove: if $\mat A$ is any matrix with constant coefficients, we may solve the Dirichlet problem using the Fourier transform.
We will still use the Fourier transform to solve the Neumann problem; however, the argument will be somewhat more involved.

Throughout this section, we will let $\widehat f$ denote the Fourier transform in $\R^n$ (not $\R^\dmn$) given by
\begin{equation*}\widehat f(\xi)=\int_{\R^n} e^{-2\pi i\xi\cdot x}\,f(x)\,dx.\end{equation*}

Let $\arr g\in L^2(\R^n)$ be an array indexed by multiindices $\gamma$ with $\abs\gamma=m-1$. 
Let $\arr \varphi=\Tr_{m-1}\varphi$ for some smooth, compactly supported function~$\varphi$. As in the definition~\eqref{C:eqn:Neumann:E} of Neumann boundary values, let $\varphi_\ell(x)=\partial_t^\ell \varphi(x,t)\big\vert_{t=0}$. 
By Plancherel's theorem,
\begin{align*}
\langle \arr g, \arr \varphi\rangle_{\R^n}
&=
	\sum_{\ell=0}^{m-1} \sum_{\substack{\gamma_\dmn=\ell}} \langle g_\gamma, \partial_\pureH^{\gamma_\pureH}\varphi_\ell\rangle_{\R^n}
=
	\sum_{\ell=0}^{m-1} \sum_{\substack{\gamma_\dmn=\ell}} 
	\int_{\R^n} \overline{\widehat{g}_\gamma(\xi)} \,(2\pi i\xi)^{\gamma_\pureH}\, \widehat\varphi_{\ell}(\xi)\,d\xi
\\&=
	\sum_{\ell=0}^{m-1}
	\int_{\R^n}\widehat\varphi_{\ell}(\xi) \sum_{\substack{\abs{\gamma_\pureH}=m-1-\ell}} \overline{\widehat{g}_\gamma(\xi)} \,(2\pi i\xi)^{\gamma_\pureH} \,d\xi
.\end{align*}
Here, if $\gamma=(\gamma_1,\gamma_2,\dots,\gamma_\dmnMinusOne,\gamma_\dmn)$, then $\gamma_\pureH=(\gamma_1,\gamma_2,\dots,\gamma_\dmnMinusOne)$.

Thus, to establish existence of solutions to the Neumann problem, it suffices to show that, for each array of functions $\{G_\ell\}_{\ell=0}^{m-1}$ that satisfy the bound
\begin{equation}
\label{C:eqn:G:fourier}
\int_{\R^n} \abs{G_\ell(\xi)}^2\,\abs{\xi}^{2\ell+2-2m}\,d\xi<\infty\end{equation}
there is some function $w$ that satisfies
\begin{equation*}
\left\{\begin{aligned}
Lw&=0  \quad\text{in }\R^\dmn_+,\\
\langle \M_{\mat A}^+ w,\arr \varphi\rangle_{\R^n}
&=
\sum_{\ell=0}^{m-1}\int_{\R^n}\widehat\varphi_{\ell}(\xi) \,G_\ell(\xi)\,d\xi,
\\
\int_0^\infty \int_{\R^n} \abs{\nabla^m \partial_t w(x,t)}^2\,t\,dx\,dt&\leq C \sum_{\ell=0}^{m-1}\int_{\R^n} \abs{G_\ell(\xi)}^2\,\abs{\xi}^{2+2\ell-2m}\,d\xi
,\\
\doublebar{\nabla^m w(\,\cdot\,,t)}_{L^2(\R^n)}&<\infty \quad \text{for some $t>0$}.
\end{aligned}\right.\end{equation*}

Because $\M_{\mat A}^+ w$ is a complicated operator that depends on the choice of coefficients $\mat A$ associated to~$L$, in the remainder of this section, we will let $\mat A$ denote a particular choice of coefficients.

Let $\Delta_\pureH$ denote the Laplacian in $\R^n$, $\Delta_\pureH=\partial_{x_1}\partial_{x_1}+\dots+\partial_{x_n}\partial_{x_n}$. We observe that
\begin{equation*}(-\Delta_\pureH)^j=(-1)^j\sum_{\abs\gamma=j} \frac{j!}{\gamma_1!\gamma_2!\dots\gamma_\dmnMinusOne!} \partial_\pureH^{2\gamma}\end{equation*}
where the sum is over multiindices in $\N^n$ (equivalently multiindices in $\N^\dmn$ with $\gamma_\dmn=0$).

Let $L$ be the operator of the form~\eqref{C:eqn:divergence} associated to the (constant) coefficients $A_{\alpha\beta}$ given by 
\begin{equation*}A_{\alpha\alpha} = \frac{\abs{\alpha_\pureH}!}{\alpha_\pureH!},
\qquad A_{\alpha\beta}=0 \text{ if }\alpha\neq\beta.\end{equation*}
We thus have that
\begin{equation}\label{C:eqn:L:special}
(-\Delta_\pureH)^j = (-1)^j\sum_{\abs{\alpha_\pureH}=j} A_{\alpha\alpha}\partial_\pureH^{2\alpha_\pureH} \quad\text{and so}\quad 
L\psi = \sum_{j=0}^m(-1)^j (-\Delta_\pureH)^{m-j}\partial_\dmn^{2j} \psi.\end{equation}
Notice that $L$ is not the polyharmonic operator $(-\Delta)^m$; however, $L$ is a constant-coefficient elliptic operator and satisfies the bound \eqref{C:eqn:elliptic:slices:strong}.

For each $1\leq k\leq m$, let $f_k:\R^n\mapsto\C$ be a function that satisfies
\begin{equation}\label{C:eqn:f:fourier}\int_{\R^n} \abs{\xi}^{2m} \abs{f_k(\xi)}^2\,d\xi <\infty.\end{equation}
Let $w$ satisfy
\begin{equation}
\label{C:eqn:w:fourier}
\widehat w(\xi,t) = \sum_{k=1}^{m} 
f_k(\xi)\exp\bigl(2\pi i \abs{\xi} e^{\pi i k/(m+1)}t\bigr)
\end{equation}
where the Fourier transform is taken only in the horizontal variables. Notice that the real part of $ie^{\pi i k/(m+1)}$ is at most $-\sin (\pi/(m+1))$, and so if $t>0$ then the exponential decays as $t\to\infty$ or $\abs\xi\to\infty$.
Then 
\begin{equation*}
\sup_{t>0}\doublebar{\nabla^m w(\,\cdot\,,t)}_{L^2(\R^n)} < \infty,
\qquad
\lim_{t\to \infty} \doublebar{\nabla^m w(\,\cdot\,,t)}_{L^2(\R^n)}=0.\end{equation*}
By formula~\eqref{C:eqn:L:special},
\begin{align*}
\widehat{Lw}(\xi,t)
&= \sum_{j=0}^m (-1)^{j}(4\pi^2\abs{\xi}^2)^{m-j}\partial_t^{2j} \widehat w(\xi,t) 
\\&= (4\pi^2\abs\xi^2)^m
\sum_{k=1}^{m} 
f_k(\xi)\exp\bigl(2\pi i \abs{\xi} e^{\pi i k/(m+1)}t\bigr)\sum_{j=0}^m 
e^{2\pi i jk/(m+1)}
.\end{align*}
Summing the geometric series, we see that $Lw=0$ in $\R^\dmn_+$.

Furthermore, by Parseval's inequality,
\begin{align*}
\int_0^\infty \int_{\R^n} \abs{\nabla^m \partial_t w(x,t)}^2\,t\,dx\,dt
&\leq
	C\sum_{j=0}^{m}\int_0^\infty t\int_{\R^n}  \abs{\xi}^{2j} \abs{ \partial_t^{m+1-j} \widehat w(\xi,t)}^2\,d\xi\,dt
.\end{align*}
By the definition~\eqref{C:eqn:w:fourier} of~$\widehat w$,
\begin{multline*}\int_0^\infty \int_{\R^n} \abs{\nabla^m \partial_t w(x,t)}^2,t,dx,dt
\\\leq
	C
	\sum_{k=1}^m
	\int_0^\infty t\int_{\R^n}  \abs{\xi}^{2m+2} \abs{f_k(\xi)}^2
	\exp(-2\beta_k  \abs{\xi} t)
	\,d\xi\,dt
\end{multline*}
where $\beta_k=2\pi\sin(\pi k/(m+1))\geq \beta_1>0$. Interchanging the order of integration and evaluating the integral in~$t$, we see that
\begin{align}
\label{C:eqn:w:norm}
\int_0^\infty \int_{\R^n} \abs{\nabla^m \partial_t w(x,t)}^2,t,dx,dt
&\leq
	C
	\sum_{k=1}^m
	\int_{\R^n}  \abs{\xi}^{2m} \abs{f_k(\xi)}^2
	\,d\xi
.\end{align}

By definition of $\mat A$ and~$\mathcal{E}$, we have that
\begin{align*}
\langle \M_{\mat A}^+ w,\arr \varphi\rangle_{\R^n}
&=
\sum_{\ell=0}^{m-1} \frac{1}{\ell!}
\sum_{\substack{\abs\alpha=m }}
\int_0^\infty \langle  A_{\alpha\alpha}\partial^\alpha w(\,\cdot\,,t), \partial^\alpha(t^\ell\mathcal{Q}_t^m\varphi_\ell)\rangle_{\R^n}\,dt
.\end{align*}
By Plancherel's theorem, and because $\mat A$ is constant,
\begin{multline*}
\langle \M_{\mat A}^+ w,\arr \varphi\rangle_{\R^n}
\\=
\sum_{\ell=0}^{m-1} \sum_{j=0}^m 
\sum_{\substack{\abs\alpha=m \\ \alpha_\perp=j}}
\frac{1}{\ell!}
\int_0^\infty \langle  A_{\alpha\alpha}(2\pi i\,\cdot\,)^{\alpha_\pureH}\partial_t^j\widehat w(\,\cdot\,,t), (2\pi i\,\cdot\,)^{\alpha_\pureH} \partial_t^j(t^\ell\widehat{\mathcal{Q}_t^m\varphi_\ell})\rangle_{\R^n}\,dt
.\end{multline*}
By definition of $A_{\alpha\alpha}$,
\begin{equation*}
\langle \M_{\mat A}^+ w,\arr \varphi\rangle_{\R^n}=
\sum_{j=0}^m \sum_{\ell=0}^{m-1} 
\frac{1}{\ell!}
\int_0^\infty \langle  (2\pi \abs{\,\cdot\,})^{2m-2j}\partial_t^j\widehat w(\,\cdot\,,t), \partial_t^j(t^\ell\widehat{\mathcal{Q}_t^m\varphi_\ell})\rangle_{\R^n}\,dt
.\end{equation*}
Recall that $\mathcal{Q}_t^m=e^{-(-t^2\Delta_\pureH)^m}$. Thus, $\widehat{\mathcal{Q}_t^m\psi}(\xi)=e^{-(4\pi^2t^2\abs{\xi}^2)^m}\widehat\psi(\xi)$, and so
\begin{multline*}
\langle \M_{\mat A}^+ w,\arr \varphi\rangle_{\R^n}
\\=
\sum_{\ell=0}^{m-1} \sum_{j=0}^m 
\frac{1}{\ell!}
\int_0^\infty \int_{\R^n}
  (2\pi \abs{\xi})^{2m-2j}\partial_t^j\widehat w(\xi,t)\,\partial_t^j(t^\ell e^{-(4\pi^2 t^2\abs{\xi}^2)^m})\widehat{\varphi}_\ell(\xi)\,d\xi\,dt
.\end{multline*}
By definition of~$w$,
\begin{multline*}
\langle \M_{\mat A}^+ w,\arr \varphi\rangle_{\R^n}
\\\begin{aligned}
&=
	\sum_{\ell=0}^{m-1}\sum_{j=0}^m  \sum_{k=1}^{m} 
	\frac{1}{\ell!}
	\int_0^\infty \int_{\R^n}
	  (2\pi \abs{\xi})^{2m-2j}f_k(\xi)\,\partial_t^j
	\exp\bigl(2\pi i \abs{\xi} e^{\pi i k/(m+1)}t\bigr)
	\\&\qquad\qquad\times\partial_t^j(t^\ell e^{-(4\pi^2 t^2\abs{\xi}^2)^m})\widehat{\varphi}_\ell(\xi)\,d\xi\,dt
\\&=
	\sum_{\ell=0}^{m-1}\sum_{j=0}^m  \sum_{k=1}^{m} 
	\frac{1}{\ell!}
	\int_0^\infty \int_{\R^n}
	  \widehat{\varphi}_\ell(\xi)\,f_k(\xi)\,(2\pi \abs{\xi})^{2m-j}\, i^j  e^{\pi i jk/(m+1)}
	\\&\qquad\qquad\qquad
	\times\exp\bigl(2\pi i \abs{\xi} e^{\pi i k/(m+1)}t\bigr)
	\partial_t^j(t^\ell e^{-(4\pi^2 t^2\abs{\xi}^2)^m})\,d\xi\,dt
.\end{aligned}\end{multline*}
We wish to change the order of integration. We must show that the integral converges absolutely; it will be technically easier to show absolute convergence after the change.

Making the change of variables $u=t\abs{\xi}$, we see that 
\begin{equation*}\int_0^\infty 
	\abs{\exp\bigl(2\pi i \abs{\xi} e^{\pi i k/(m+1)}t\bigr)
	\partial_t^j(t^\ell e^{-(4\pi^2 t^2\abs{\xi}^2)^m})}\,dt = C_{j,k,\ell}\abs{\xi}^{j-\ell-1}.\end{equation*}
But by assumption on~$f$, and because $\varphi_\ell$ is smooth and compactly supported.
\begin{equation*}\int_{\R^n} \abs{\widehat{\varphi}_\ell(\xi)}\,\abs{f_k(\xi)}\,(2\pi \abs{\xi})^{2m-j}\,C_{j,k,\ell}\abs{\xi}^{j-\ell-1}\,d\xi<\infty.\end{equation*}
Thus we may change the order of integration to see that
\begin{multline*}
\langle \M_{\mat A}^+ w,\arr \varphi\rangle_{\R^n}
\\=
	\sum_{\ell=0}^{m-1}\sum_{j=0}^m  \sum_{k=1}^{m} 
	\frac{1}{\ell!}
	\int_{\R^n}\widehat{\varphi}_\ell(\xi)\,f_k(\xi)
	(2\pi \abs{\xi})^{2m-j}i^j e^{\pi i jk/(m+1)}
	\\\times\int_0^\infty 
	\exp\bigl(2\pi i \abs{\xi} e^{\pi i k/(m+1)}t\bigr)
	\partial_t^j(t^\ell e^{-(4\pi^2 t^2\abs{\xi}^2)^m})\,dt\,d\xi
.\end{multline*}

We will need a precies formula for (not a bound on) the second integral. We will obtain it by integrating by parts in~$t$. If $0\leq J\leq m$, then $\lim_{t\to 0^+} \partial_t^J (t^\ell e^{-\alpha t^{2m}})=0$ unless $J=\ell$, in which case the limit is~$\ell!$.
Thus, if $j\geq 1+ \ell$ then
\begin{multline*}
	\int_0^\infty \exp\bigl(2\pi i \abs{\xi} e^{\pi i k/(m+1)}t\bigr)
	\,\partial_t^j(t^\ell e^{-(4\pi^2t^2\abs{\xi}^2)^m})\,dt
\\\begin{aligned}
&=
	\int_0^\infty (-1)^j\partial_t^j\exp\bigl(2\pi i \abs{\xi} e^{\pi i k/(m+1)}t\bigr)
	\,(t^\ell e^{-(4\pi^2t^2\abs{\xi}^2)^m})\,dt
\\&\qquad+
 \lim_{t\to 0^+}
	(-1)^{j-\ell} \partial_t^{j-1-\ell} \exp\bigl(2\pi i \abs{\xi} e^{\pi i k/(m+1)}t\bigr)\ell!
\end{aligned}\end{multline*}
so
\begin{multline*}
	\int_0^\infty \exp\bigl(2\pi i \abs{\xi} e^{\pi i k/(m+1)}t\bigr)
	\,\partial_t^j(t^\ell e^{-(4\pi^2t^2\abs{\xi}^2)^m})\,dt
\\\begin{aligned}
&=
	\bigl(-2\pi i \abs{\xi})^j e^{\pi i jk/(m+1)}
	\int_0^\infty \exp\bigl(2\pi i \abs{\xi} e^{\pi i k/(m+1)}t\bigr)
	\,(t^\ell e^{-(4\pi^2t^2\abs{\xi}^2)^m})\,dt
\\&\qquad-
	  \bigl(-2\pi i \abs{\xi} e^{\pi i k/(m+1)}\bigr)^{j-1-\ell}\ell!
.\end{aligned}\end{multline*}
If $j\leq \ell$ then we have a very similar formula without the second term. Thus,
\begin{multline*}
\langle \M_{\mat A}^+ w,\arr \varphi\rangle_{\R^n}
\\ \begin{aligned} &=
	\sum_{\ell=0}^{m-1} \sum_{k=1}^{m} 
	\frac{1}{\ell!}
	\int_{\R^n}\widehat{\varphi}_\ell(\xi)\,f_k(\xi)
	(2\pi \abs{\xi})^{2m} \sum_{j=0}^m  e^{2\pi i jk/(m+1)}
	\\&\qquad\qquad\qquad\times
	\int_0^\infty \exp\bigl(2\pi i \abs{\xi} e^{\pi i k/(m+1)}t\bigr)
	\,(t^\ell e^{-(4\pi^2t^2\abs{\xi}^2)^m})\,dt\,d\xi
	\\&\quad-
	\sum_{\ell=0}^{m-1} \sum_{k=1}^{m} 
	\int_{\R^n}\widehat{\varphi}_\ell(\xi)\,f_k(\xi)
	(2\pi \abs{\xi})^{2m-1-\ell} i^{1+\ell} 	
	\sum_{j=\ell+1}^m   e^{\pi i (2j-1-\ell)k/(m+1)}
	 \,d\xi
.\end{aligned}
\end{multline*}
Summing our two geometric series, we see that
\begin{multline*}
\langle \M_{\mat A}^+ w,\arr \varphi\rangle_{\R^n}
\\=
	-\sum_{\ell=0}^{m-1} \sum_{k=1}^{m} 
	\int_{\R^n}\widehat{\varphi}_\ell(\xi)\,f_k(\xi)
	(2\pi \abs{\xi})^{2m-1-\ell} i^{\ell} 
	\frac{2\sin({\pi  (1+\ell)k/(m+1)}) }{e^{2\pi i k/(m+1)}-1}
	 \,d\xi
.\end{multline*}
Recall that, given functions $G_\ell$, 
we wish to find functions $f_k$ such that 
\begin{equation*}\langle \M_{\mat A}^+ w,\arr \varphi\rangle_{\R^n}
=\sum_{\ell=0}^{m-1}
	\int_{\R^n}\widehat\varphi_{\ell}(\xi) G_\ell(\xi)\,d\xi
\end{equation*}
and such that 
\begin{equation*}
\sum_{k=1}^m \int_{\R^n} \abs{f_k(\xi)}^2\abs{\xi}^{2m}\,d\xi
\leq C\sum_{\ell=0}^{m-1}
\int_{\R^n} \abs{G_\ell(\xi)}^2\,\abs{\xi}^{2+2\ell-2m}\,d\xi.\end{equation*}
Thus, it suffices to find functions  $f_k$ that satisfy the  bound \eqref{C:eqn:f:fourier} and the equations
\begin{equation*}(2\pi \abs{\xi})^{-m+1+\ell}G_\ell(\xi)=- 2i^{\ell} \sum_{k=1}^{m} 
	\frac{(2\pi \abs{\xi})^{m} f_k(\xi)}{e^{2\pi i k/(m+1)}-1}
	{\sin({\pi  (1+\ell)k/(m+1)}) }.\end{equation*}
As is well known in, for example, the theory of the discrete Fourier transform, the $m\times m$ matrix $M=\begin{pmatrix}M_{Lk}\end{pmatrix}_{L,k=1}^m$ whose entries are given by $M_{Lk}=\sin({\pi Lk/(m+1)})$ is invertible. Thus, given $G_\ell$, we may find functions~$f_k$; if the functions $G_\ell$ satisfy the bound~\eqref{C:eqn:G:fourier}, then the functions $f_k$ satisfy the bound~\eqref{C:eqn:f:fourier}, as desired.

\section{Invertibility of layer potentials and boundary value problems}
\label{C:sec:invertible}

There is a deep connection between well posedness of boundary value problems and invertibility of layer potentials.
The classic method of layer potentials states that if $\M_{\mat A}^+\D^{\mat A}$ is surjective $\DD\mapsto \NN$, then solutions to the Neumann problem with boundary values in $\NN$ exist. In \cite{Ver84}, Verchota proved a result (for harmonic functions, but the argument generalizes easily) going in the other direction: if solutions to the Neumann problem are unique in both $\R^\dmn_+$ and $\R^\dmn_-$, then $\M_{\mat A}^+\D^{\mat A}$ is one-to-one. The converses to these results for second order operators were proven in \cite{BarM13,BarM16A}, and the generalization to the higher order case was established in \cite{Bar17p}. 

We will summarize the relevant results of \cite{Bar17p} in Section~\ref{C:sec:invertible:known} and apply them in Section~\ref{C:sec:invertible:application}.

\subsection{Known results}
\label{C:sec:invertible:known}

Let  $\XX^+$ and $\XX^-$ be two spaces of functions (or equivalence classes of functions) defined in $\R^\dmn_+$ and $\R^\dmn_-$, respectively, and assume that if $u\in \XX^\pm$ then $\nabla^m u$ is locally integrable.
Let $\DD$ and $\NN$ be two spaces of equivalence classes of functions or distributions defined on $\R^n=\partial\R^\dmn_+$. 

Then we have the following theorem.

\begin{thm}[{\cite{Bar17p}}]
\label{C:thm:invertible}
Suppose that $L$ is an elliptic operator of order $2m$ associated with coefficients $\mat A$ that satisfy the ellipticity conditions \eqref{C:eqn:elliptic} and~\eqref{C:eqn:elliptic:bounded}.
Suppose that the following conditions are valid.

\begin{enumerate}
\item If $ u\in \XX^\pm$ and $L u=0$ in~$\R^\dmn_\pm$, then $\Tr_{m-1}^\pm u\in\DD$ and 
$\M_{\mat A}^\pm  u\in \NN$. 
\item The single layer potential $\s^{L}$ is bounded $\NN\mapsto \XX^+$ and $\NN\mapsto \XX^-$.
\item The double layer potential $\D^{\mat A}$ is bounded $\DD\mapsto\XX^+$ and $\DD\mapsto\XX^-$.
\item If $\arr g\in\NN$, then we have  the jump relations 
\begin{align*}
\Tr_{m-1}^+\s^{L} \arr g -\Tr_{m-1}^-\s^{L} \arr g
	&=0
,\\
 \M_{\mat A}^+ \s^{L} \arr g
+\M_{\mat A}^- \s^{L}\arr g
	&=\arr g
.\end{align*}
\item If $\arr f\in\DD$, then we have the jump relations
\begin{align*}
\Tr_{m-1}^+\D^{\mat A} \arr f -\Tr_{m-1}^-\D^{\mat A} \arr f
	&=-\arr f
,\\
\M_{\mat A}^+ \D^{\mat A} \arr f + \M_{\mat A}^- \D^{\mat A}\arr f 
	&=0
.\end{align*}
\item If $ u \in \XX^\pm$ and $L u=0$ in~$\R^\dmn_\pm$, then we have the Green's formulas
\begin{equation*} u = \mp\D^{\mat A} (\Tr_{m-1}^\pm u) + \s^{L} (\M_{\mat A}^\pm u)\quad\text{in }\XX^\pm.
\end{equation*}
\end{enumerate}

Then $\M_{\mat A}^\pm\D^{\mat A}$ is surjective $\DD\mapsto\NN$ if and only if, for every $\arr g\in\NN$, there exists a $u_+\in \XX^+$ and a $u_-\in\XX^-$ such that $\pm\M_{\mat A}^\pm u_\pm = \arr g$. (In this case there is some $\arr f\in\DD$ such that $u_\pm=\D^{\mat A} \arr f$.)

Furthermore, the bound $\doublebar{\arr f}_\DD \leq C\doublebar{\M_{\mat A}\D^{\mat A}\arr f}_\NN$ is valid for all $\arr f\in\DD$ if and only if the two bounds $\doublebar{u_+}_{\XX^+}\leq C \doublebar{\M_{\mat A}^+ u_+}_{\NN}$ and $\doublebar{u_-}_{\XX^+}\leq C \doublebar{\M_{\mat A}^+ u_-}_{\NN}$ are valid for all $u_\pm\in\XX^\pm$ with $Lu_\pm=0$ in $\R^\dmn_\pm$.
\end{thm}

Notice that all results must be checked in both the upper and lower half-spaces; this is because of the use of the jump relations. We remark that, by considering the change of variables $(x,t)\mapsto (x,-t)$, all of the results of Sections~\ref{C:sec:trace}, \ref{C:sec:neumann:limit}, \ref{C:sec:green} and~\ref{C:sec:rellich} are valid in the lower half-space as well as the upper half-space.

The jump relations are well known in the second order case. To check conditions~(4) and~(5) of Theorem~\ref{C:thm:invertible}, it will be useful to have the following fact.
\begin{lem}[{\cite{Bar17p}}]\label{C:lem:solution:jump}
Let $L$ be as in Theorem~\ref{C:thm:invertible}.
Let $\arr f\in \dot W\!A^2_{m-1,1/2}(\R^n)$ and let $\arr g\in (\dot W\!A^2_{m-1,1/2}(\R^n))^*$. Then the jump relations of conditions~(4) and~(5) are valid.
\end{lem}

In order to prove Theorem~\ref{C:thm:well-posed:1}, we will need an adjoint relation for layer potentials; again, this result is well known in the second order case and may be easily generalized to the higher order case.
\begin{lem}[{\cite{Bar17p}}]\label{C:lem:solution:adjoint}
Let $L$ and $\mat A$ be as in Theorem~\ref{C:thm:invertible}.
Let $\mat A^*$ be the adjoint matrix, that is, $A^*_{\alpha\beta}=\overline{A_{\beta\alpha}}$. Let $L^*$ be the associated elliptic operator.
Then we have the adjoint relations
\begin{align}
\label{C:eqn:neumann:D:dual}
\langle \arr \varphi, \M_{\mat A}^+ \D^{\mat A} \arr f\rangle_{\R^n} 
&= \langle\M_{\mat A^*}^+ \D^{\mat A^*}  \arr \varphi, \arr f\rangle_{\R^n} 
,\\
\label{C:eqn:dirichlet:S:dual}
\langle \arr \gamma, \Tr_{m-1} \s^{L} \arr g\rangle_{\R^n} 
&= \langle \Tr_{m-1} \s^{L^*} \arr \gamma, \arr g\rangle_{\R^n} 
\end{align}
for all $\arr f$, $\arr \varphi\in \dot W\!A^2_{m-1,1/2}(\R^n)$ and all   $\arr g$, $\arr \gamma\in (\dot W\!A^2_{m-1,1/2}(\R^n))^*$.
\end{lem}

\subsection{Proofs of the main theorems}
\label{C:sec:invertible:application}

In order to apply Theorem~\ref{C:thm:invertible}, we must show that the boundary and solution spaces of Theorems~\ref{C:thm:well-posed:2} and~\ref{C:thm:well-posed:1} satisfy the given conditions. We will do so in the following two lemmas.

\begin{lem}\label{C:lem:invertible:1} If $L$ is as in Theorem~\ref{C:thm:invertible}, then the spaces
\begin{gather*}
\XX^\pm =\biggl\{v:\int_{\R^\dmn_\pm} \abs{\nabla^m v(x,t)}^2\abs{t}\,dx\,dt<\infty,\>
\sup_{\pm t>0}\doublebar{\nabla^{m-1} v(\,\cdot\,,t)}_{L^2(\R^n)}<\infty\biggr\}, \\
\DD = \dot W\!A^2_{m-1,0}(\R^n),\qquad
\NN = (\dot W\!A^2_{m-1,1}(\R^n))^*
\end{gather*}
satisfy the conditions of Theorem~\ref{C:thm:invertible}.
\end{lem}
\begin{proof}
Condition~(1) follows from Theorems~\ref{C:thm:Dirichlet:1} and~\ref{C:thm:Neumann:1}. Conditions~(2) and~(3) follow from Theorem~\ref{C:thm:square:rough}. The jump relations of conditions (4) and~(5) are true for $\arr f$ and $\arr g$ in dense subspaces of $\dot W\!A^2_{m-1,0}(\R^n)$ and $(\dot W\!A^2_{m-1,1}(\R^n))^*$; conditions~(1--3) imply that conditions~(4) and~(5) are true by density. Condition~(6) is valid by Theorem~\ref{C:thm:green}.
\end{proof}

\begin{lem}\label{C:lem:invertible:2} If $L$ is as in Theorem~\ref{C:thm:invertible}, then the spaces
\begin{gather*}
\XX^\pm =\biggl\{w:\int_{\R^\dmn_\pm} \abs{\nabla^m \partial_t w(x,t)}^2\abs{t}\,dx\,dt<\infty,\>\doublebar{\nabla^m w(\,\cdot\,,1)}_{L^2(\R^n)}<\infty\biggr\}, \\
\DD = \dot W\!A^2_{m-1,1}(\R^n),\qquad
\NN = (\dot W\!A^2_{m-1,0}(\R^n))^*
\end{gather*}
satisfy the conditions of Theorem~\ref{C:thm:invertible}.
\end{lem}

\begin{proof} Condition~(1) follows from Theorems~\ref{C:thm:Dirichlet:2} and~\ref{C:thm:Neumann:2}. Conditions (2) and~(3) follow from Theorem~\ref{C:thm:square}. The jump relations of conditions (4) and~(5) are true for $\arr f$ and $\arr g$ in dense subspaces of $\dot W\!A^2_{m-1,1}(\R^n)$ and $(\dot W\!A^2_{m-1,0}(\R^n))^*$; Conditions~(1--3) imply that Conditions~(4) and~(5) are true by density.
Condition~(6) is valid by Theorem~\ref{C:thm:green}.
\end{proof}

We now prove our main theorems.

\begin{proof}[Proof of Theorem~\ref{C:thm:well-posed:2}]
By Theorem~\ref{C:thm:invertible}, Lem\-ma~\ref{C:lem:invertible:2} and Theorem~\ref{C:thm:neumann:unique}, if $\mat A$ is as in Theorem~\ref{C:thm:rellich}, then the operators $\M_{\mat A}^\pm\D^{\mat A}$ 
satisfies the estimate
\begin{equation*}\doublebar{\arr f}_{\dot W\!A^2_{m-1,1}(\R^n)}\leq C \doublebar{\M_{\mat A}^\pm \D^{\mat A}\arr f}_{(\dot W\!A^2_{m-1,0}(\R^n))^*}.\end{equation*}
We need only show that this operator is surjective to complete the proof of Theorem~\ref{C:thm:well-posed:2}.

By Theorem~\ref{C:thm:special} there is a coefficient matrix $\mat A_0$ such that solutions to the Neumann problem exist, and so $\M_{\mat A_0}^+\D^{\mat A_0}$ is onto $\dot W\!A^2_{m-1,1}(\R^n)\mapsto {(\dot W\!A^2_{m-1,0}(\R^n))^*}$.

Choose some $\mat A$. Let $\mat A_s=(1-s)\mat A_0+s\mat A$. Observe that $\mat A_s$ is self-adjoint, bounded and elliptic, uniformly in $0\leq s\leq 1$. 

Let $\M_s=\M_{\mat A_s}^+\D^{\mat A_s}$. Then there exists some constants $C_0$ and $C_1$ depending on the ellipticity constants of $\mat A$ and $\mat A_0$ such that 
\begin{equation*}
\frac{1}{C_0}\doublebar{\arr f}_{\dot W\!A^2_{m-1,1}(\R^n)}
\leq \doublebar{\M_s\arr f}_{(\dot W\!A^2_{m-1,0}(\R^n))^*}\leq C_1\doublebar{\arr f}_{\dot W\!A^2_{m-1,1}(\R^n)}\end{equation*}
for all $\arr f\in{\dot W\!A^2_{m-1,1}(\R^n)}$ and all $0\leq s\leq 1$.

By analytic perturbation theory, if $0\leq r\leq s\leq 1$ and $\abs{s-r}$ is small enough (depending only on~$C_1$), then
\begin{equation*}\doublebar{\M_s - \M_r} \leq C\abs{s-r}\end{equation*}
where, again, the constant $C$ depends only on~$C_1$, and where the given norm is the operator norm ${\dot W\!A^2_{m-1,1}(\R^n)}\mapsto {(\dot W\!A^2_{m-1,0}(\R^n))^*}$.

Suppose that $\M_r$ is onto (and thus is bijective). Its inverse has operator norm at most~$C_0$. Let $\abs{s-r}=1/N$ for some integer~$N$; we may choose $N$ large enough, depending only on $C_0$ and~$C_1$, such that 
\begin{equation*}\doublebar{\M_s - \M_r} \leq \frac{1}{2C_0}\leq \frac{1}{2\doublebar{\M_r^{-1}}}.\end{equation*}
Now, consider the operator
\begin{equation*}\M
= \sum_{j=0}^\infty \M_r^{-1}[(\M_r-\M_s)\M_r^{-1}]^j.\end{equation*}
The sum converges to a bounded operator defined on all of ${(\dot W\!A^2_{m-1,0}(\R^n))^*}$. But
\begin{align*}\M_s \M 
&= (\M_s-\M_r)\M+\M_r\M 
\\&= \sum_{j=0}^\infty [(\M_r-\M_s)\M_r^{-1}]^j-\sum_{j=0}^\infty [(\M_r-\M_s)\M_r^{-1}]^{j+1}
=I\end{align*}
is the identity, and so $\M=\M_s^{-1}$ and $\M_s$ is surjective as well. Thus, since $\M_0$ is surjective, by working in small steps we see that $\M_1=\M_{\mat A}\D^{\mat A}$ is surjective for any bounded self-adjoint elliptic matrix~$\mat A$.

Thus, $\M_{\mat A}^+\D^{\mat A}$ is invertible from ${\dot W\!A^2_{m-1,1}(\R^n)}$ to $ {(\dot W\!A^2_{m-1,0}(\R^n))^*} $, and so the Neumann problem \eqref{C:eqn:neumann:regular} is well posed for coefficients as in Theorem~\ref{C:thm:well-posed:2}. This completes the proof.
\end{proof}

\begin{proof}[Proof of Theorem~\ref{C:thm:perturbation}] Suppose that $\mat A_0$ is as in Theorem~\ref{C:thm:perturbation}. Then by Theorem~\ref{C:thm:invertible} and Lemma~\ref{C:lem:invertible:1} or~\ref{C:lem:invertible:2}, 
\[\M_{\mat A_0}^+\D^{\mat A_0}:\dot W\!A^2_{m-1,0}\mapsto (\dot W\!A^2_{m-1,1})^*,\qquad \M_{\mat A_0}^+\D^{\mat A_0}:\dot W\!A^2_{m-1,1}\mapsto(\dot W\!A^2_{m-1,0})^*\]
are invertible mappings. If $\mat A$ is $t$-independent and sufficiently close to $\mat A_0$, then $\mat A$ is also elliptic, and so by the bounds~\eqref{C:eqn:Neumann:D:1} and~\eqref{C:eqn:Neumann:D:2} $\M_{\mat A}^+\D^{\mat A}$ is bounded as mappings between the same two pairs of spaces. Thus, by analytic perturbation theory, if $\mat A$ is sufficiently close to $\mat A_0$, then $\M_{\mat A}^+\D^{\mat A}$ is also invertible, and so the Neumann problem \eqref{C:eqn:neumann:regular} or~\eqref{C:eqn:neumann:rough} is well posed, as desired.
\end{proof}

\begin{proof}[Proof of Theorem~\ref{C:thm:well-posed:1}]
Observe that by the duality relation \eqref{C:eqn:neumann:D:dual} (true in dense subsets of the relevant spaces), 
\[\M_{\mat A}^+\D^{\mat A}:\dot W\!A^2_{m-1,0}(\R^n)\mapsto (\dot W\!A^2_{m-1,1}(\R^n))^*\]
is invertible if and only if 
\[\M_{\mat A^*}^+\D^{\mat A^*}:\dot W\!A^2_{m-1,1}(\R^n)\mapsto (\dot W\!A^2_{m-1,0}(\R^n))^*\]
is invertible. Thus, Theorems~\ref{C:thm:well-posed:2} and~\ref{C:thm:well-posed:1} are equivalent.
\end{proof}

\begin{rmk}\label{C:rmk:HofKMP15B} Under the conditions of Theorem~\ref{C:thm:invertible}, invertibility of $\Tr_{m-1}^+\s^L$  is equivalent to well posedness of the Dirichlet problem. See \cite{Bar17p}.

Thus, as in the proof of Theorem~\ref{C:thm:well-posed:1} (using the duality relation \eqref{C:eqn:dirichlet:S:dual} in place of the relation~\eqref{C:eqn:neumann:D:dual}), for elliptic $t$-independent coefficients, well posedness of the Dirichlet problem, with coefficients $\mat A$, boundary data in $\dot W\!A^2_{m-1,0}(\R^n)$, and solutions as in Lemma~\ref{C:lem:invertible:1}, implies well posedness of the Dirichlet problem with coefficients $\mat A^*$, boundary data in $\dot W\!A^2_{m-1,1}(\R^n)$, and solutions as in Lemma~\ref{C:lem:invertible:2}. (The Dirichlet problem with boundary data in $\dot W\!A^p_{m-1,1}$ rather than $\dot W\!A^p_{m-1,0}$ is known in the theory as the ``Dirichlet regularity'' problem.) 

The main result of \cite{HofKMP15B} is that for second order $t$-independent operators, well posedness of the Dirichlet problem with coefficients $\mat A$ and boundary data in $L^p(\R^n)$ implies well posedness of the regularity problem with coefficients $\mat A^*$ and boundary data in $\dot W^{p'}_1(\R^n)$ for $1/p+1/p'=1$, provided $2-\varepsilon<p<\infty$. Indeed they prove that solutions may be represented as the single layer potential of some appropriate input function. Notice that the trace results of \cite{BarHM17pB} (in particular Theorem~\ref{C:thm:Neumann:1}) are essential to the argument presented here, and that those theorems were proven using many ideas from \cite{HofKMP15B}; the approach described here may be thought of as another way of formulating the arguments of \cite{HofKMP15B}.
\end{rmk}


\begin{thebibliography}{HKMP15b}

\bibitem[AA11]{AusA11}
Pascal Auscher and Andreas Axelsson, \emph{Weighted maximal regularity
  estimates and solvability of non-smooth elliptic systems {I}}, Invent. Math.
  \textbf{184} (2011), no.~1, 47--115. \MR{2782252}

\bibitem[AAA{\etalchar{+}}11]{AlfAAHK11}
M.~Angeles Alfonseca, Pascal Auscher, Andreas Axelsson, Steve Hofmann, and
  Seick Kim, \emph{Analyticity of layer potentials and {$L^2$} solvability of
  boundary value problems for divergence form elliptic equations with complex
  {$L^\infty$} coefficients}, Adv. Math. \textbf{226} (2011), no.~5,
  4533--4606. \MR{2770458}

\bibitem[AAH08]{AusAH08}
Pascal Auscher, Andreas Axelsson, and Steve Hofmann, \emph{Functional calculus
  of {D}irac operators and complex perturbations of {N}eumann and {D}irichlet
  problems}, J. Funct. Anal. \textbf{255} (2008), no.~2, 374--448. \MR{2419965
  (2009h:35079)}

\bibitem[AAM10]{AusAM10}
Pascal Auscher, Andreas Axelsson, and Alan McIntosh, \emph{Solvability of
  elliptic systems with square integrable boundary data}, Ark. Mat. \textbf{48}
  (2010), no.~2, 253--287. \MR{2672609 (2011h:35070)}

\bibitem[Agm57]{Agm57}
Shmuel Agmon, \emph{Multiple layer potentials and the {D}irichlet problem for
  higher order elliptic equations in the plane. {I}}, Comm. Pure Appl. Math
  \textbf{10} (1957), 179--239. \MR{0106323 (21 \#5057)}

\bibitem[AHL{\etalchar{+}}02]{AusHLMT02}
Pascal Auscher, Steve Hofmann, Michael Lacey, Alan McIntosh, and Ph.
  {Tcha\-mit\-chian}, \emph{The solution of the {K}ato square root problem for
  second order elliptic operators on {$\mathbb{R}^n$}}, Ann. of Math. (2)
  \textbf{156} (2002), no.~2, 633--654. \MR{1933726 (2004c:47096c)}

\bibitem[AHMT01]{AusHMT01}
Pascal Auscher, Steve Hofmann, Alan McIntosh, and Philippe Tchamitchian,
  \emph{The {K}ato square root problem for higher order elliptic operators and
  systems on {$\mathbb{R}^n$}}, J. Evol. Equ. \textbf{1} (2001), no.~4,
  361--385, Dedicated to the memory of Tosio Kato. \MR{1877264 (2003a:35046)}

\bibitem[AM14]{AusM14}
Pascal Auscher and Mihalis Mourgoglou, \emph{Boundary layers, {R}ellich
  estimates and extrapolation of solvability for elliptic systems}, Proc. Lond.
  Math. Soc. (3) \textbf{109} (2014), no.~2, 446--482. \MR{3254931}

\bibitem[AQ00]{AusQ00}
P.~Auscher and M.~Qafsaoui, \emph{Equivalence between regularity theorems and
  heat kernel estimates for higher order elliptic operators and systems under
  divergence form}, J. Funct. Anal. \textbf{177} (2000), no.~2, 310--364.
  \MR{1795955 (2001j:35057)}

\bibitem[AR12]{AusR12}
Pascal Auscher and Andreas Ros{\'e}n, \emph{Weighted maximal regularity
  estimates and solvability of nonsmooth elliptic systems, {II}}, Anal. PDE
  \textbf{5} (2012), no.~5, 983--1061. \MR{3022848}

\bibitem[AS14]{AusS14p}
Pascal Auscher and Sebastian Stahlhut, \emph{\emph{A priori} estimates for
  boundary value elliptic problems via first order systems},
  \href{http://arxiv.org/abs/1403.5367v2}{\tt 1403.5367v2 [math.CA]}, jun 2014.

\bibitem[Bar]{Bar17p}
Ariel Barton, \emph{Layer potentials for general linear elliptic systems}, in
  preparation.

\bibitem[Bar13]{Bar13}
\bysame, \emph{Elliptic partial differential equations with almost-real
  coefficients}, Mem. Amer. Math. Soc. \textbf{223} (2013), no.~1051, vi+108.
  \MR{3086390}

\bibitem[Bar16]{Bar16}
\bysame, \emph{Gradient estimates and the fundamental solution for higher-order
  elliptic systems with rough coefficients}, Manuscripta Math. \textbf{151}
  (2016), no.~3-4, 375--418. \MR{3556825}

\bibitem[BHMa]{BarHM17pA}
Ariel Barton, Steve Hofmann, and Svitlana Mayboroda, \emph{Bounds on layer
  potentials with rough inputs for higher order elliptic equations}, in
  preparation.

\bibitem[BHMb]{BarHM17pB}
\bysame, \emph{{D}irichlet and {N}eumann boundary values of solutions to higher
  order elliptic equations}, in preparation.

\bibitem[BHMc]{BarHM17pD}
\bysame, \emph{Nontangential estimates and the {N}eumann problem for higher
  order elliptic equations}, in preparation.

\bibitem[BHMd]{BarHM15p}
\bysame, \emph{Square function estimates on layer potentials for higher-order
  elliptic equations}, Math. Nachr., to appear.

\bibitem[BM13]{BarM13}
Ariel Barton and Svitlana Mayboroda, \emph{The {D}irichlet problem for higher
  order equations in composition form}, J. Funct. Anal. \textbf{265} (2013),
  no.~1, 49--107. \MR{3049881}

\bibitem[BM16a]{BarM16B}
\bysame, \emph{Higher-order elliptic equations in non-smooth domains: a partial
  survey}, Harmonic Analysis, Partial Differential Equations, Complex Analysis,
  Banach Spaces, and Operator Theory (Volume 1). Celebrating Cora Sadosky's
  life, Association for Women in Mathematics Series, vol.~4, Springer-Verlag,
  2016, pp.~55--121.

\bibitem[BM16b]{BarM16A}
\bysame, \emph{Layer potentials and boundary-value problems for second order
  elliptic operators with data in {B}esov spaces}, Mem. Amer. Math. Soc.
  \textbf{243} (2016), no.~1149, v+110. \MR{3517153}

\bibitem[Cam80]{Cam80}
S.~Campanato, \emph{Sistemi ellittici in forma divergenza. {R}egolarit\`a
  all'interno}, Qua\-der\-ni. [Publications], Scuola Normale Superiore Pisa,
  Pisa, 1980. \MR{668196 (83i:35067)}

\bibitem[CG83]{CohG83}
Jonathan Cohen and John Gosselin, \emph{The {D}irichlet problem for the
  biharmonic equation in a {$C^{1}$} domain in the plane}, Indiana Univ. Math.
  J. \textbf{32} (1983), no.~5, 635--685. \MR{711860 (85b:31004)}

\bibitem[CG85]{CohG85}
\bysame, \emph{Adjoint boundary value problems for the biharmonic equation on
  {$C^1$} domains in the plane}, Ark. Mat. \textbf{23} (1985), no.~2, 217--240.
  \MR{827344 (88d:31006)}

\bibitem[Dah80]{Dah80A}
Bj{\"o}rn E.~J. Dahlberg, \emph{Weighted norm inequalities for the {L}usin area
  integral and the nontangential maximal functions for functions harmonic in a
  {L}ipschitz domain}, Studia Math. \textbf{67} (1980), no.~3, 297--314.
  \MR{592391 (82f:31003)}

\bibitem[DJK84]{DahJK84}
Bj{\"o}rn E.~J. Dahlberg, David~S. Jerison, and Carlos~E. Kenig, \emph{Area
  integral estimates for elliptic differential operators with nonsmooth
  coefficients}, Ark. Mat. \textbf{22} (1984), no.~1, 97--108. \MR{735881
  (85h:35021)}

\bibitem[DK87]{DahK87}
Bj{\"o}rn E.~J. Dahlberg and Carlos~E. Kenig, \emph{Hardy spaces and the
  {N}eumann problem in {$L^p$} for {L}aplace's equation in {L}ipschitz
  domains}, Ann. of Math. (2) \textbf{125} (1987), no.~3, 437--465. \MR{890159
  (88d:35044)}

\bibitem[DKPV97]{DahKPV97}
B.~E.~J. Dahlberg, C.~E. Kenig, J.~Pipher, and G.~C. Verchota, \emph{Area
  integral estimates for higher order elliptic equations and systems}, Ann.
  Inst. Fourier (Grenoble) \textbf{47} (1997), no.~5, 1425--1461. \MR{1600375
  (98m:35045)}

\bibitem[DKV86]{DahKV86}
B.~E.~J. Dahlberg, C.~E. Kenig, and G.~C. Verchota, \emph{The {D}irichlet
  problem for the biharmonic equation in a {L}ipschitz domain}, Ann. Inst.
  Fourier (Grenoble) \textbf{36} (1986), no.~3, 109--135. \MR{865663
  (88a:35070)}

\bibitem[FMM98]{FabMM98}
Eugene Fabes, Osvaldo Mendez, and Marius Mitrea, \emph{Boundary layers on
  {S}obolev-{B}esov spaces and {P}oisson's equation for the {L}aplacian in
  {L}ipschitz domains}, J. Funct. Anal. \textbf{159} (1998), no.~2, 323--368.
  \MR{1658089 (99j:35036)}

\bibitem[HKMP15a]{HofKMP15B}
Steve Hofmann, Carlos Kenig, Svitlana Mayboroda, and Jill Pipher, \emph{The
  regularity problem for second order elliptic operators with complex-valued
  bounded measurable coefficients}, Math. Ann. \textbf{361} (2015), no.~3-4,
  863--907. \MR{3319551}

\bibitem[HKMP15b]{HofKMP15A}
\bysame, \emph{Square function\slash non-tangential maximal function estimates
  and the {D}irichlet problem for non-symmetric elliptic operators}, J. Amer.
  Math. Soc. \textbf{28} (2015), no.~2, 483--529. \MR{3300700}

\bibitem[HMM15a]{HofMayMou15}
Steve Hofmann, Svitlana Mayboroda, and Mihalis Mourgoglou, \emph{Layer
  potentials and boundary value problems for elliptic equations with complex
  {$L^\infty$} coefficients satisfying the small {C}arleson measure norm
  condition}, Adv. Math. \textbf{270} (2015), 480--564. \MR{3286542}

\bibitem[HMM15b]{HofMitMor15}
Steve Hofmann, Marius Mitrea, and Andrew~J. Morris, \emph{The method of layer
  potentials in {$L^p$} and endpoint spaces for elliptic operators with
  {$L^\infty$} coefficients}, Proc. Lond. Math. Soc. (3) \textbf{111} (2015),
  no.~3, 681--716. \MR{3396088}

\bibitem[Jaw77]{Jaw77}
Bj{\"o}rn Jawerth, \emph{Some observations on {B}esov and {L}izorkin-{T}riebel
  spaces}, Math. Scand. \textbf{40} (1977), no.~1, 94--104. \MR{0454618 (56
  \#12867)}

\bibitem[JK81]{JerK81B}
David~S. Jerison and Carlos~E. Kenig, \emph{The {N}eumann problem on
  {L}ipschitz domains}, Bull. Amer. Math. Soc. (N.S.) \textbf{4} (1981), no.~2,
  203--207. \MR{598688 (84a:35064)}

\bibitem[Ken94]{Ken94}
Carlos~E. Kenig, \emph{Harmonic analysis techniques for second order elliptic
  boundary value problems}, CBMS Regional Conference Series in Mathematics,
  vol.~83, Published for the Conference Board of the Mathematical Sciences,
  Washington, DC, 1994. \MR{1282720 (96a:35040)}

\bibitem[KP93]{KenP93}
Carlos~E. Kenig and Jill Pipher, \emph{The {N}eumann problem for elliptic
  equations with nonsmooth coefficients}, Invent. Math. \textbf{113} (1993),
  no.~3, 447--509. \MR{1231834 (95b:35046)}

\bibitem[KP95]{KenP95}
\bysame, \emph{The {N}eumann problem for elliptic equations with nonsmooth
  coefficients. {II}}, Duke Math. J. \textbf{81} (1995), no.~1, 227--250
  (1996), A celebration of John F. Nash, Jr. \MR{1381976 (97j:35021)}

\bibitem[KR09]{KenR09}
Carlos~E. Kenig and David~J. Rule, \emph{The regularity and {N}eumann problem
  for non-symmetric elliptic operators}, Trans. Amer. Math. Soc. \textbf{361}
  (2009), no.~1, 125--160. \MR{2439401 (2009k:35050)}

\bibitem[Liz60]{Liz60}
P.~I. Lizorkin, \emph{Boundary properties of functions from ``weight''
  classes}, Soviet Math. Dokl. \textbf{1} (1960), 589--593. \MR{0123103 (23
  \#A434)}

\bibitem[May05]{May05}
Svitlana Mayboroda, \emph{The {P}oisson problem on {L}ipschitz domains},
  ProQuest LLC, Ann Arbor, MI, 2005, Thesis (Ph.D.)--University of
  Missouri-Columbia. \MR{2709638}

\bibitem[MM04]{MayMit04A}
Svitlana Mayboroda and Marius Mitrea, \emph{Sharp estimates for {G}reen
  potentials on non-smooth domains}, Math. Res. Lett. \textbf{11} (2004),
  no.~4, 481--492. \MR{2092902 (2005i:35059)}

\bibitem[MM13a]{MitM13B}
Irina Mitrea and Marius Mitrea, \emph{Boundary value problems and integral
  operators for the bi-{L}aplacian in non-smooth domains}, Atti Accad. Naz.
  Lincei Rend. Lincei Mat. Appl. \textbf{24} (2013), no.~3, 329--383.
  \MR{3097019}

\bibitem[MM13b]{MitM13A}
\bysame, \emph{Multi-layer potentials and boundary problems for higher-order
  elliptic systems in {L}ipschitz domains}, Lecture Notes in Mathematics, vol.
  2063, Springer, Heidelberg, 2013. \MR{3013645}

\bibitem[Nad63]{Nad63}
A.~Nadai, \emph{Theory of {F}low and {F}racture of {S}olids}, vol.~II,
  McGraw-Hill, 1963.

\bibitem[PV95]{PipV95B}
Jill Pipher and Gregory~C. Verchota, \emph{Dilation invariant estimates and the
  boundary {G}\aa rding inequality for higher order elliptic operators}, Ann.
  of Math. (2) \textbf{142} (1995), no.~1, 1--38. \MR{1338674 (96g:35052)}

\bibitem[Ros13]{Ros13}
Andreas Ros{\'e}n, \emph{Layer potentials beyond singular integral operators},
  Publ. Mat. \textbf{57} (2013), no.~2, 429--454. \MR{3114777}

\bibitem[She07]{She07B}
Zhongwei Shen, \emph{The {$L^p$} boundary value problems on {L}ipschitz
  domains}, Adv. Math. \textbf{216} (2007), no.~1, 212--254. \MR{2353255
  (2009a:35064)}

\bibitem[Ver84]{Ver84}
Gregory Verchota, \emph{Layer potentials and regularity for the {D}irichlet
  problem for {L}aplace's equation in {L}ipschitz domains}, J. Funct. Anal.
  \textbf{59} (1984), no.~3, 572--611. \MR{769382 (86e:35038)}

\bibitem[Ver90]{Ver90}
\bysame, \emph{The {D}irichlet problem for the polyharmonic equation in
  {L}ipschitz domains}, Indiana Univ. Math. J. \textbf{39} (1990), no.~3,
  671--702. \MR{1078734 (91k:35073)}

\bibitem[Ver96]{Ver96}
Gregory~C. Verchota, \emph{Potentials for the {D}irichlet problem in
  {L}ipschitz domains}, Potential theory---{ICPT} 94 ({K}outy, 1994), de
  Gruyter, Berlin, 1996, pp.~167--187. \MR{1404706 (97f:35041)}

\bibitem[Ver05]{Ver05}
\bysame, \emph{The biharmonic {N}eumann problem in {L}ipschitz domains}, Acta
  Math. \textbf{194} (2005), no.~2, 217--279. \MR{2231342 (2007d:35058)}

\bibitem[Zan00]{Zan00}
Daniel~Z. Zanger, \emph{The inhomogeneous {N}eumann problem in {L}ipschitz
  domains}, Comm. Partial Differential Equations \textbf{25} (2000), no.~9-10,
  1771--1808. \MR{1778780 (2001g:35056)}

\end{thebibliography}

\newcommand{\etalchar}[1]{$^{#1}$}
\providecommand{\bysame}{\leavevmode\hbox to3em{\hrulefill}\thinspace}
\providecommand{\MR}{\relax\ifhmode\unskip\space\fi MR }
\providecommand{\MRhref}[2]{%
  \href{http://www.ams.org/mathscinet-getitem?mr=#1}{#2}
}
\providecommand{\href}[2]{#2}

\end{document}